\documentclass[11pt]{amsart}

\usepackage{amsthm,amssymb,amstext,amscd,amsfonts,amsbsy,amsxtra,latexsym,amsmath,color,comment,fancybox}
\usepackage{fullpage}
\usepackage[english]{babel}
\usepackage[latin1]{inputenc}
\usepackage{cancel}
\usepackage{mathrsfs}
\usepackage{enumitem}
\allowdisplaybreaks

\newcommand{\field}[1]{\mathbb{#1}}
\newcommand{\N}{\field{N}}
\newcommand{\Z}{\field{Z}}
\newcommand{\R}{\field{R}}
\newcommand{\C}{\field{C}}

\renewcommand{\H}{\mathbb{H}}
\newcommand{\sgn}{\mathrm{sgn}}
\newcommand{\SL}{\operatorname{SL}}

\newcommand{\im}{\operatorname{Im}}
\newcommand{\re}{\operatorname{Re}}

\newcommand{\bea}{\begin{eqnarray}}
\newcommand{\eea}{\end{eqnarray}}
\newcommand{\be}{\begin {equation}}
\newcommand{\ee}{\end{equation}}

\numberwithin{equation}{section}
\newtheorem{theorem}{Theorem}
\numberwithin{theorem}{section}
\newtheorem{lemma}[theorem]{Lemma}
\newtheorem{proposition}[theorem]{Proposition}
\newtheorem{corollary}[theorem]{Corollary}

\theoremstyle{remark}
\newtheorem*{remark}{Remark}
\newtheorem*{remarks}{Remarks}
\newtheorem*{definition}{Definition}

\renewenvironment{proof}[1][Proof]{\begin{trivlist} \item[\hskip \labelsep {\bfseries #1:}]}{\qed\end{trivlist}}

\begin{document}
\title{A problem of Petersson about weight 0 meromorphic modular forms}
\author{Kathrin Bringmann}
\address{Mathematical Institute\\University of
Cologne\\ Weyertal 86-90 \\ 50931 Cologne \\Germany}
\email{kbringma@math.uni-koeln.de}
\author{Ben Kane}
\address{Department of Mathematics\\ University of Hong Kong\\ Pokfulam, Hong Kong}
\email{bkane@maths.hku.hk}
\date{\today}
\begin{abstract}
In this paper, we provide an explicit construction of weight $0$ meromorphic modular forms.  Following work of Petersson, we build these via Poincar\'e series.  There are two main aspects of our investigation which differ from his approach.  Firstly, the naive definition of the Poincar\'e series diverges and one must analytically continue via Hecke's trick.  Hecke's trick is further complicated in our situation by the fact that the Fourier expansion does not converge everywhere due to singularities in the upper half-plane so it cannot 
solely be used to analytically continue the functions.  To explain the second difference, we recall that Petersson constructed linear combinations from a family of meromorphic functions which are modular if a certain principal parts condition is satisfied.  In contrast to this, we construct linear combinations from a family of non-meromorphic modular forms, known as polar harmonic Maass forms, which are meromorphic whenever the principal parts condition is satisfied.
\end{abstract}
\thanks{The research of the first author was supported by the Alfried Krupp Prize for Young University Teachers of the Krupp foundation and the research leading to these results has received funding from the European Research Council under the European Union's Seventh Framework Programme (FP/2007-2013) / ERC Grant agreement n. 335220 - AQSER.  The research of the second author was supported by grant project numbers 27300314 and 17302515 of the Research Grants Council.}
\maketitle

\section{Introduction and statement of results}
A special case of the Riemann--Roch Theorem gives a sufficient and necessary condition for the existence of meromorphic modular forms with prescribed principal parts.  Although this implies the existence of meromorphic modular forms with certain prescribed principal parts, it unfortunately fails to explicitly 
produce  them.  Using Poincar\'e series, Petersson achieved the goal of an explicit construction for negative weight forms in \cite{Pe1,Pe2}, but did not cover the case of weight $0$ considered in this paper.  This paper deals with difficulties caused by divergence of the naive Poincar\'e series and also views the problem from a different perspective than Petersson's.  In particular, the question is placed into the context of a larger space of non-meromorphic modular forms, allowing the usage of modern techniques to avoid some of the difficulties of Petersson's method.

To give a flavor of the differences between these methods, we delve a little deeper into the history of Petersson's related work in \cite{Pe1}.  The relevant meromorphic modular forms are constructed via a family of two-variable meromorphic Poincar\'e series if the corresponding group only has one cusp and the weight is negative.  The Poincar\'e series have positive weight in one variable by construction, but the other variable can be used as a gateway between the space of positive-weight forms and their dual negative-weight counterparts.  
In this way, the existence of meromorphic modular forms implied by the special case of the Riemann--Roch Theorem considered in \cite{Pe1} may be viewed as a sufficient and necessary condition for certain linear combinations of Poincar\'e series to satisfy (negative weight) modularity in the second variable.  Following this logic, Satz 3 of \cite{Pe1} provides an explicit version of the existence implied by Riemann--Roch.

Petersson then pointed out two remaining tasks: the first  one pertains to including general subgroups.  He later achieved this in Satz 16 of \cite{Pe2}, with an explicit representation of the forms given in (56) of \cite{Pe2}.  He then asked whether there is a generalization to weight $0$.  In this paper, we settle Petersson's second question by viewing it in a larger space of the so-called polar harmonic Maass forms,  generalizations of Bruinier-Funke harmonic Maass forms \cite{BF}.  These are modular objects which are no longer meromorphic but which are instead annihilated by the  hyperbolic Laplacian.  In this larger space, principal parts may essentially be chosen arbitrarily and the principal part condition of Riemann--Roch translates into a condition which determines whether a polar harmonic Maass form with a given principal part is meromorphic.  The subspace of harmonic Maass forms have appeared in a number of recent applications.  For example, Zwegers \cite{Zwegers} recognized the mock theta functions, introduced by Ramanujan in his last letter to Hardy, as ``holomorphic parts'' of half-integral weight harmonic Maass forms.   Generating functions for central values and derivatives of quadratic twists of weight 2 modular $L$-functions were later proven to be weight $1/2$ harmonic Maass forms by Bruinier and Ono \cite{BrO}.  Duke, Imamo$\overline{\text{g}}$lu, and T\'oth \cite{DITRamanujan} used weight $2$ harmonic Maass forms to evaluate the inner products between modular functions.

In this paper, we instead investigate Maass forms which are also allowed to grow at points in the upper 
half-plane.
  Such forms are of growing interest because they yield lifts of meromorphic modular forms, which occur throughout various applications. Just to mention a few examples, Duke and Jenkins \cite{DJ} studied traces of meromorphic modular forms and such functions are also of importance for constructing canonical lifts \cite{AGOR, Gu}.

We now return to the main question addressed in this paper, the classification via principal parts of those polar harmonic Maass forms which are meromorphic.  To state the results, we require some notation. Denote by $\mathcal{S}_N$ a set of inequivalent cusps of $\Gamma_0(N)$ and for each $\varrho\in \mathcal{S}_N$,  $\ell_{\varrho}$ is the cusp width of $\varrho$.  For $\mathfrak{z}\in\H$, we furthermore denote $\omega_{\mathfrak{z}}=\omega_{\mathfrak{z},N}:=
\# \Gamma_{\mathfrak{z}}
$, where $\Gamma_{\mathfrak{z},N}$ is the stabilizer group of $\mathfrak{z}$ 
in $\operatorname{PSL}_2(\Z)\cap\Gamma_0(N)$.  Throughout, we write $z=x+iy$, $\mathfrak{z}=\mathfrak{z}_1+i\mathfrak{z}_2$, and $\tau_d=u_d+iv_d$.  There is a well-known family of polar harmonic Maass forms $\mathcal{P}_{2-2k,n, N}^{\varrho}(z)$ with $n\in-\N$, each of which has principal part $e^{2\pi i n z/\ell_{\varrho}}$ as $z$ approaches $\varrho\in\mathcal{S}_N$ and no other singularities in $\H\cup \mathcal{S}_N$. 
Thus for an explicit construction of weight $0$ forms, it only remains to build forms with singularities in the upper half-plane.  In particular, 
 the main step in this paper
 is to use Hecke's trick to explicitly define a family of polar harmonic Maass forms $\mathcal{Y}_{0,n,N}(\mathfrak{z},z)$ in \eqref{eqn:Ydef} with principal parts $X_{\mathfrak{z}}^{n}(z)$ at $z=\mathfrak{z}$ and no other singularities in $\Gamma_0(N)\backslash\H\cup \mathcal{S}_N$, where
$$
X_{\mathfrak{z}}(z):=\frac{z-\mathfrak{z}}{z-\overline{\mathfrak{z}}}.
$$

For this task, we take inspiration from Petersson in two different directions; firstly, an argument in \cite{Pe} is augmented 
to extend Hecke's trick to the case when the Poincar\'e series have poles, yielding $\mathcal{Y}_{0,-1,N}$, and secondly differential operators constructed in \cite{Pe2} 
are applied to $\mathcal{Y}_{0,-1,N}$ to construct the family $\mathcal{Y}_{0,n,N}$.  The construction of these Poincar\'e series is of independent interest; together with the functions $\mathcal{P}_{0,n,N}^{\varrho}$, they form a basis of the space of polar harmonic Maass forms with explicit principal parts and we plan to study them further in future research. Hence the principal parts of any linear combination of $\mathcal{P}_{2-2k,n,N}^{\varrho}$ and $\mathcal{Y}_{2-2k,n,N}$ may be quickly determined.   Moreover, up to constant functions if $k=1$, all weight $2-2k\leq 0$ polar harmonic Maass forms, and hence in particular all meromorphic modular forms, may be expressed as linear combinations of $\mathcal{P}_{2-2k,n, N}^{\varrho}$ and $\mathcal{Y}_{2-2k,n,N}$.  In this language, the necessary and sufficient condition 
implied by (a special case of) Riemann--Roch is equivalent to determining whether a given linear combination of these polar harmonic Poincar\'e series is meromorphic.

\begin{theorem}\label{Peterssonch}
For $\Gamma_0(N)$-inequivalent points $\tau_1,\dots,\tau_{r}\in\H$  and $k\geq 1$, there exists a weight $2-2k$ meromorphic modular form on $\Gamma_0(N)$ with principal parts at each cusp $\varrho$ equal to $\sum_{n<0} a_{\varrho}(n) e^{\frac{2\pi i nz}{\ell_{\varrho}}}$ and principal parts in $\H$ given by
$$
\sum_{d=1}^{r}\left(z-\overline{\tau_{d}}\right)^{2k-2}\sum_{\substack{n \equiv k-1 \pmod{\omega_{\tau_{d}}}\\ n<0}} b_{\tau_{d}}(n) X_{\tau_{d}}^{n}(z)
$$
if and only if, for every cusp form $g\in S_{2k}(N)$, the constants $a_{\varrho}(n)$ and $b_{\tau_{d}}(n)$ satisfy the principal part condition
\begin{equation}\label{eqn:Peterssonch}
\frac{1}{2\pi i }\sum_{\varrho\in\mathcal{S}_N} \sum_{n>0}a_\varrho(-n) a_{g,\varrho}(n) +\sum_{d=1}^{r} \frac{1}{2iv_{d} \omega_{\tau_{d}}}\sum_{\substack{n\equiv 1-k\pmod{\omega_{\tau_{d}}}\\ n>0}}b_{\tau_{d}}(-n)a_{ g,\tau_{d}}(n-1)=0,
\end{equation}
where $a_{g,\varrho}(n)$ is the $n$th Fourier coefficient of $g$ at the cusp $\varrho$ and $a_{g,\mathfrak{z}}(n)$ is the $n$th coefficient in the elliptic expansion of $g$ around $\mathfrak{z}$.
Specifically, the weight $2-2k$ polar harmonic Maass form
\begin{equation}\label{eqn:genpolar}
\sum_{\varrho\in\mathcal{S}_N} \sum_{n<0} a_{\varrho}(n) \mathcal{P}_{2-2k,n,N}^{\varrho}(z) + \sum_{d=1}^{r} \sum_{\substack{n\equiv k-1\pmod{\omega_{\tau_{d}}}\\ n<0}} b_{\tau_{d}}(n) \mathcal{Y}_{2-2k,n,N}(\tau_{d},z)
\end{equation}
is a meromorphic modular form if and only if \eqref{eqn:Peterssonch} is satisfied.
\end{theorem}

\begin{remarks}
\noindent

\noindent
\begin{enumerate}[leftmargin=*]
\item Note that \eqref{eqn:Peterssonch} only has to be checked for $\dim_{\C} S_{2k} (N)$  many cusp forms.
\item For genus $0$ subgroups there is a simpler direct proof of Theorem \ref{Peterssonch}, using explicit basis elements.
\item An alternative approach for constructing basis elements is to average 
coefficients in the elliptic expansion of a Maass form. For a good description of such types of Poincar\'e series, see \cite{IO}, 
while the case of forms with singularities at the cusps  \cite{Hejhal,Niebur}.
\end{enumerate}
\end{remarks}
By computing the Fourier coefficients of the basis elements, Theorem \ref{Peterssonch} yields the Fourier expansions of all meromorphic modular forms.
\begin{corollary}\label{cor:Fourier}
Suppose that $\tau_1,\dots,\tau_{r}\in\H$ are $\Gamma_0(N)$-inequivalent and that there exists a meromorphic modular function  $f$ on $\Gamma_0(N)$ with principal parts at each cusp $\varrho$ equal to $\sum_{n<0} a_{\varrho}(n) e^{\frac{2\pi i nz}{\ell_{\varrho}}}$ and principal parts in $\H$ given by
$$
\sum_{d=1}^{r}\sum_{\substack{n \equiv 0 \pmod{\omega_{\tau_{d}}}\\ n<0}} b_{\tau_{d}}(n) X_{\tau_{d}}^{n}(z).
$$
The function $f$ has a Fourier expansion which is valid for $y$ sufficiently large (depending on $v_1,\dots,v_d$).  For $m\in\N$, the $m$th Fourier coefficient of $f$ is given by
$$
\sum_{\varrho\in\mathcal{S}_N} \sum_{n<0} a_{\varrho}(n)
c\left(\mathcal{P}_{0,n,N}^{\varrho},m\right)
 + \sum_{d=1}^{r} \sum_{\substack{n\equiv 0\pmod{\omega_{\tau_{d}}}\\ n<0}} b_{\tau_{d}}(n) c\left(\mathcal{Y}_{0,n,N}(\tau_{d},\cdot),m\right),
$$
with $c(g,m)$ the $m$th coefficient of $g$.  The coefficients of $f$ are explicitly given, independent of $\mathcal{P}_{0,n,N}^{\varrho}$ and $\mathcal{Y}_{0,n,N}$, in Theorem \ref{thm:Fourier}.
 \end{corollary}
\begin{remarks}
\noindent

\noindent
\begin{enumerate}[leftmargin=*]
\item
While the functions $\mathcal{Y}_{2-2k,n,N}$ are useful for Theorem \ref{Peterssonch}, for $k>1$ there are forms constructed from $\mathcal{Y}_{2-2k,-1,N}$ by applying another natural differential operator, known as the raising operator, which yield Fourier expansions closely resembling the expansions given by Hardy and Ramanujan for the reciprocal of the weight $6$ Eisenstein series.  The authors \cite{BKFourier} applied this method to obtain the Fourier expansions of negative-weight meromorphic modular forms.
\item
The expansions of $F$ at the other cusps can easily be derived from Lemma \ref{lem:alpha/gammalim} and the definition \eqref{eqn:Ydef} of $\mathcal{Y}_{0,n,N}$.
\item
The result for $k>1$ was proven by Petersson (see (69) of \cite{Pe2}).  Furthermore, in this case, one can recognize the $m$th coefficient of $\mathcal{Y}_{2-2k,-1,N}$ as a constant multiple of the weight $2k$ Poincar\'e series with principal part $e^{-2\pi i m \mathfrak{z}}$ at $i\infty$. Using this, one can write the Fourier coefficients of a given weight $2-2k$ meromorphic modular form as the image of an operator acting on weight $2k$ meromorphic modular forms.  The explicit version of Corollary \ref{cor:Fourier} given in Theorem \ref{thm:Fourier} yields an analogous operator on weight $2$ meromorphic modular forms, but we do not work out the details here.
\end{enumerate}
\end{remarks}

In this paper, we do not extensively investigate the properties of the functions $
\mathfrak{z}\mapsto \mathcal{Y}_{0,n,N}(\mathfrak{z},z)
$.
  However, there are a number of properties related to weight $2$ meromorphic modular forms which are worth noting.  One can show that $
\mathfrak{z}\mapsto \mathcal{Y}_{0,-1,N}(\mathfrak{z},z)
$ satisfies weight $2$ modularity.  Taking the trace over $z=\tau_Q$ for roots $\tau_Q$ of inequivalent binary quadratic forms $Q$ of discriminant $D<0$ yields a weight 2 version of the functions 
$$
f_{k,D}(\mathfrak{z}):=\sum_{\substack{Q=[a,b,c]\\ b^2-4ac=D}}Q(\mathfrak{z},1)^{-k}.
$$
The analogous functions $f_{k,D}$, with $D>0$, are weight $2k$ cusp forms which were investigated by Zagier \cite{ZagierRQ} and played an important role in Kohnen and Zagier's construction of a kernel function \cite{KohnenZagier} for the Shimura \cite{Shimura} and Shintani \cite{Shintani} lifts.  Kohnen and Zagier then used this kernel function to prove the non-negativity of twisted central $L$-values of cusp forms.  As shown by  Bengoechea \cite{Bengoechea}, the $f_{k,D}$ functions, with $D<0$, are meromorphic modular forms of weight $2k$ with poles of order $k$ at $\mathfrak{z}=\tau_Q$ and which  decay like cusp forms towards the cusps.  The authors \cite{BKcycle} proved that inner products of these meromorphic modular forms with other meromorphic modular forms lead to a new class of modular objects, the first case of which is a polar harmonic Maass form, and that they also appear as theta lifts, which was generalized to vector-valued forms by Zemel \cite{Zemel}.  Noting the applications of the $f_{k,D}$ functions for $D>0$, it may be interesting to investigate the properties of $\mathcal{Y}_{0,-1,N}$ if $z$ is a CM-point.  In particular, the periods of these forms are of interest because they have geometric applications. This will be 
studied further in future research.

We do however investigate one aspect of the properties in the $\mathfrak{z}$-variable here.  
Recall that if $f$ is a weight $2$ meromorphic modular form, then $f(\mathfrak{z})d\mathfrak{z}$ is a meromorphic differential, and, in the classical language, we say that the differential is of the \begin{it}first kind\end{it} if $f$ is holomorphic, it is of the \begin{it}second kind\end{it} if $f$ is not holomorphic but the residue vanishes at every pole, and of the \begin{it}third kind\end{it} if all of the poles are simple (for further information about the connection between differentials and meromorphic modular forms, see page 182 of \cite{Pe6}).  One can use $\mathcal{Y}_{0,-1,N}$ to construct differentials of all three kinds.

\begin{theorem}\label{thm:differentials}
\noindent

\noindent
\begin{enumerate}[leftmargin=*]
\item[\rm(1)]
As a function of $\mathfrak{z}$, $\mathcal{Y}_{0,-1,N}(\mathfrak{z},z)$ corresponds to a differential of the third kind.
\item[\rm(2)] The function $z\mapsto\xi_{0,z}(\mathcal{Y}_{0,-1,N}(\mathfrak{z},z))$ corresponds to a differential of the first kind (as a function of $z$).  In other words, $\xi_{0,z}(\mathcal{Y}_{0,-1,N}(\mathfrak{z},z))$ is a cusp form. 
\item[\rm(3)] The function $z\mapsto D_z\left(\mathcal{Y}_{0,-1,N}(\mathfrak{z},z)\right)$ corresponds to a differential of the second kind.
\end{enumerate}
\end{theorem}

  Although we only investigate the Fourier coefficients of $\mathcal{Y}_{0,n,N}$ in the $z$-variable, the same techniques can be applied to compute the Fourier coefficients in the $\mathfrak{z}$-variable.  Noting the connections to differentials given above, it might be interesting to explicitly determine the behavior towards each cusp in order to compute the differential of the third kind associated with $\mathcal{Y}_{0,-1,N}$.

 Another direction future research may take involves the question of whether $\mathcal{Y}_{0,n,N}$ may be constructed in a similar manner for more general subgroups.  The methods in this paper can indeed be extended to obtain more general subgroups.  The main difficulty lies in proving analytic continuation of the Kloosterman zeta functions for these subgroups.  Finally we want to mention that the properties as functions of $\mathfrak{z}$ are of interest. 

The paper is organized as follows.  In Section \ref{sec:prelim}, we give background on polar harmonic Maass forms and in particular harmonic Maass forms.  We then determine the shape of the elliptic expansions of polar harmonic Maass forms.  In Section \ref{sec:analcont} we use Hecke's trick together with a splitting of \cite{Pe} to analytically continue two-variable elliptic Poincar\'e series $
y^{2k-1}
\Psi_{2k,N}(\mathfrak{z},z)$ to include $k=1$.  After that, we determine the properties of the analytic continuation as a function of $z$ in Section \ref{sec:explicit}.  In particular, $y\Psi_{2,N}(\mathfrak{z},z)$ yields $\mathcal{Y}_{0,-1,N}(\mathfrak{z},z)$, up to a constant multiple of the non-holomorphic weight $2$ Eisenstein series $\widehat{E}_2(\mathfrak{z})$ and is invariant under the action of $\Gamma_0(N)$ as functions of $z$.  The Fourier expansions of $y\Psi_{2,N}(\mathfrak{z},z)$ at each cusp are then computed in Section \ref{sec:cusps} and an explicit basis of all polar harmonic Maass forms is constructed.  Finally, in Section \ref{sec:mainproof}, we extend a pairing of Bruinier and Funke \cite{BF} to obtain a pairing between weight $0$ polar harmonic Maass forms and weight $2$ cusp forms.  For a fixed polar harmonic Maass form, this pairing is trivial if and only if the polar harmonic Maass form is a meromorphic modular form.  We conclude the paper by computing the pairing between \eqref{eqn:genpolar} and every cusp form in $S_2(N)$, yielding Theorems \ref{Peterssonch} and \ref{thm:Fourier}.

\section*{Acknowledgement}
The authors thank Paul Jenkins, Steffen L\"obrich, Ken Ono, Martin Raum, Olav Richter, Larry Rolen, and Shaul Zemel for helpful comments on an earlier version of this paper.

\section{Preliminaries}\label{sec:prelim}
In this section, we define the space of polar harmonic Maass forms and some of its distinguished subspaces and then 
determine the shape of  the elliptic expansions of such forms.  For background on the well-studied subspace of harmonic Maass forms, which were introduced by Bruinier and Funke, we refer the reader for example to \cite{BF,Fay,Hejhal,Niebur}.

\subsection{Polar harmonic Maass forms}
We are now ready to define the modular objects which are central for this paper.
\begin{definition}
For $\kappa \in\Z_{<1}$ and $N\in\N$, a \begin{it}polar harmonic Maass form\end{it} of weight $\kappa$ on $\Gamma_0(N)$ is a function $F:\H\to\overline{\C} := \C \cup \{\infty\}$ which is real analytic outside of a discrete set and which satisfies the following conditions:
\noindent

\noindent
\begin{enumerate}[leftmargin=*]
\item
For every $M=\left(\begin{smallmatrix}a&b\\ c&d\end{smallmatrix}\right)\in\Gamma_0(N)$, we have $F|_{\kappa}M=F$, where
\begin{equation*}
F(z)|_{\kappa}M  :=j(M,z)^{-\kappa}F(Mz)
\end{equation*}
with $j(M,z):=cz+d$.
\item
The function $F$ is annihilated by the \begin{it}weight $\kappa$ hyperbolic Laplacian\end{it}
$$
\Delta_{\kappa}:=-y^2\left(\frac{\partial^2}{\partial x^2}+\frac{\partial^2}{\partial y^2}\right)+i\kappa y\left(\frac{\partial}{\partial x}+i\frac{\partial}{\partial y}\right).
$$
\item
For all $\mathfrak{z}\in \H$, there exists $n\in\N_0$ such that $(z-\mathfrak{z})^n F(z)$ is bounded in a neighborhood of $\mathfrak{z}$.
\item
The function $F$ grows at most linear exponentially towards cusps of $\Gamma_0(N)$.
\end{enumerate}
If one allows in (2) a general eigenvalue under $\Delta_{\kappa}$, then one obtains a \begin{it}polar Maass form.\end{it} Moreover \begin{it}weak Maass forms\end{it} are polar Maass forms which do not have any singularities in $\H$.
\end{definition}

We denote the space of all weight $\kappa$ polar harmonic Maass forms on $\Gamma_0(N)$ by $\mathcal{H}_{\kappa}(N)$.  An important subspace of $\mathcal{H}_{\kappa}(N)$ is obtained by noting that $
\Delta_{\kappa}
$ splits as
\begin{equation}\label{eqn:Deltasplit}
\Delta_\kappa=-\xi_{2-\kappa}\circ \xi_\kappa,
\end{equation}
where $\xi_{\kappa}:=2iy^{\kappa} \overline{\frac{\partial}{\partial \overline{z}}}$.  If $F$ satisfies weight $\kappa$ modularity, then $\xi_{\kappa}(F)$ is modular of weight $2-\kappa$.  The kernel of $\xi_{\kappa}$ is the subspace $\mathcal{M}_{\kappa}(N)$ 
of meromorphic modular forms, while one sees from the decomposition \eqref{eqn:Deltasplit} that if $F\in \mathcal{H}_{\kappa}(N)$, then $\xi_{\kappa}(F)\in \mathcal{M}_{2-\kappa}(N)$.  It is thus natural to consider the subspace $\mathcal{H}_{\kappa}^{\operatorname{cusp}}(N)\subseteq \mathcal{H}_{\kappa}(N)$ consisting of those $F$ for which $\xi_{\kappa}(F)$ is a cusp form.  The space $\mathcal{H}_{\kappa}^{\operatorname{cusp}}(N)$ decomposes into the direct sum of the subspace $H_{\kappa}^{\operatorname{cusp}}(N)$ of harmonic Maass forms which map to cusp forms under $\xi_{\kappa}$ and the subspace $\H_{\kappa}^{\operatorname{cusp}}(N)$ of polar harmonic Maass forms whose singularities in $\H$ are all poles and which are bounded towards all cusps.
In addition to $\xi_{\kappa}$, further operators on polar Maass forms appear in another natural splitting $\Delta_{\kappa}= -R_{\kappa-2}\circ L_{\kappa}$. Here $R_{\kappa}:=2i \frac{\partial}{\partial z} + \kappa y^{-1}$ is the Maass \begin{it}raising operator\end{it} and $L_{\kappa}:=-2i y^2\frac{\partial}{\partial \overline{z}}$ is the Maass \begin{it}lowering operator.\end{it}  The raising (resp. lowering) operator sends weight $\kappa$ polar Maass forms to weight $\kappa+2$ (resp. $\kappa-2$) polar Maass forms with different eigenvalues. 

Every $F\in \mathcal{H}_{\kappa}^{\operatorname{cusp}}(N)$ has a Fourier expansion around each $\varrho\in\mathcal{S}_N$ given by
$$
F_{\varrho}(z)=\sum_{n\gg -\infty}a_{F,\varrho}(n)e^{\frac{2\pi i nz}{\ell_{\varrho}}} + \sum_{n < 0}b_{F,\varrho}(n) \Gamma\left(1-\kappa,\frac{4\pi|n|y}{\ell_{\varrho}}\right)e^{\frac{2\pi i nz}{\ell_{\varrho}}},
$$
where $F_{\varrho}:=F|_{2-\kappa}M_{\varrho}\text{ with } M_{\varrho}^{-1}
\varrho =i\infty$ ($M_{\varrho}\in\SL_2(\Z)$) and $\Gamma(s,y):=\int_{y}^{\infty} t^{s-1}e^{-t}dt$ is the \begin{it}incomplete gamma function\end{it}.
The first sum is the \begin{it}meromorphic part\end{it} of $F_{\varrho}$ and the second sum is the \begin{it}non-meromorphic part\end{it} of $F_{\varrho}$.
  We call $\sum_{n< 0} a_{F,\varrho}(n)e^{\frac{2\pi i nz}{
\ell_{\varrho}}}$ the \begin{it}principal part of $F$ at $\varrho$.\end{it} Furthermore, for each $\mathfrak{z}\in\H$, there exist finitely many $c_n\in\C$ such that, in a neighborhood around $\mathfrak{z}$,
$$
F(z) - (z-\overline{\mathfrak{z}})^{-\kappa}\sum_{n<0} c_n X_{\mathfrak{z}}^{n}(z)=O(1).
$$
We call $(z-\overline{\mathfrak{z}})^{-\kappa}\sum_{n<0} c_n X_{\mathfrak{z}}^{n}(z)$ the \begin{it}principal part of $F$ at $\mathfrak{z}$.\end{it}

\subsection{Construction of weak Maass forms}
We next recall a well-known construction of Maass--Poincar\'e series, which constitute a basis of $H_{\kappa}^{\operatorname{cusp}}(N)$.  Let $\Gamma_\varrho$ be the stabilizer of $\varrho$ in $\Gamma_0(N)$.  Moreover, for $s\in\C$ and $w\in\R\setminus\{0\}$, set
\[
\mathscr{M}_{\kappa,s}(w):=|w|^{-\frac{\kappa}{2}}M_{\text{sgn}(w)\frac{\kappa}{2}, s-\frac12}(|w|),
\]
where $M_{\nu, \mu}$ is the usual $M$-Whittaker function. Then let
\[
\phi_{\kappa,s}(z):=\mathscr{M}_{\kappa,s}(4\pi y)e^{2\pi i x}
\]
and define for $m \in -\N$ the \textit{Maass Poincar\'e series} associated to the cusp $\varrho$ 
\[
\mathcal{P}_{\kappa, m,N,s}^{\varrho}(z):=\sum_{M\in\Gamma_\varrho\backslash\Gamma_0(N)}\frac{\phi_{\kappa,s}\left(\frac{m}{\ell_{\varrho}}M_\varrho Mz\right)}{j(M_{\varrho}M, z)^{\kappa}}.
\]
Note that this function converges absolutely for $\sigma:=\mathrm{Re}(s)>1$ and is a weak Maass form with eigenvalue $s(1-s)+(\kappa^2-2\kappa)/4.$

We are particularly interested in the harmonic Maass forms $\mathcal{P}_{\kappa, m,N,s}^{\varrho}$ arising from $s=1-\kappa/2$.  To state their Fourier expansions, we 
require some notation.  Denote by $I_{\nu}$ the $\nu$th $I$-Bessel function and define for $\varrho=\alpha/\gamma\in\mathcal{S}_N$ and $n,j\in\Z$, the Kloosterman sum
\begin{equation}\label{eqn:Kagdef}
K_{\alpha, \gamma} \left( n,j;c \right):=
\sum_{\substack{a \pmod{\ell_{\varrho} c}\\ d \pmod{c} \\ ad \equiv 1 \pmod {c}\\ c\equiv - a\alpha\gamma \pmod{N}}} e_{\ell_{\varrho} c} \left( n \ell_{\varrho} d+ja\right)
\end{equation}
with $e_\ell(x):=e^{\frac{2\pi i x}{\ell}}$.  
Finally, 
$\delta_{\varrho,\mu}=1$ if $\varrho$ is equivalent to the cusp $\mu$ modulo $\Gamma_0(N)$ and $0$ otherwise.
\begin{theorem}\label{thm:weakMaass}
For every $n\in -\N$, the function $
\mathcal{P}_{\kappa,n,N,s}^{\varrho}
$ has an analytic continuation to $s=1$.  Moreover, for $\kappa\leq 0$,
$$
\mathcal{P}_{\kappa,n,N}^{\varrho}(z):=\frac1{(1-\kappa)!}\mathcal{P}_{\kappa,
n,
 N,1-\frac{\kappa}{2}}^{\varrho}\left(z\right)\in H_{\kappa}^{\operatorname{cusp}}(N).
$$
For $\varrho=\alpha/\gamma$, the meromorphic part of its Fourier expansion at $i\infty$ is given by
\begin{multline*}
 \delta_{\varrho,\infty} e^{2\pi i nz}+ (2\pi)^{2-\kappa}\ell_{\varrho}^{\kappa-1} |n|^{1-\kappa}\sum_{c\geq 1} \frac{K_{\alpha,\gamma}(n,0;c)}{c^{3-\kappa}}\\
\hfill + 2\pi\left(\frac{ |n|}{\ell_{\varrho}}\right)^{\frac{1-\kappa}{2}} \sum_{j\geq 1}j^{\frac{\kappa-1}{2}}  \sum_{c\geq 1} \frac{K_{\alpha,\gamma}(n,j;c)}{c}I_{1-\kappa}\left(\frac{4\pi}{c} \sqrt{\frac{|n|j}{\ell_{\varrho}}}\right) e^{2\pi i jz}.
\end{multline*}
Its principal part at $\mu\in \mathcal{S}_N$ is given by $\delta_{\mu,\varrho} e^{\frac{2\pi inz}{\ell_{\varrho}}}.$
\end{theorem}
\begin{remark}
The Fourier expansion of $\mathcal{P}_{\kappa,n,N}^{\varrho}$ was explicitly computed in Theorem 1.1 of \cite{BO}.  Note that there are two small typos in the formula in \cite{BO}; the condition $(ad,c)=1$ in (1.11) of \cite{BO} should be replaced by $ad\equiv 1\pmod{c}$ and the power of the cusp width $t_{\mu}$ (denoted $\ell_{\varrho}$ in this paper) in (1.15) should be $1/2-k/2$ instead of $-1/2-k/2$ 
(with $k$ written as $\kappa$ here).  
\end{remark}

\subsection{Elliptic expansions of Maass forms}\label{sec:ell}
In this section, we determine the elliptic expansions of polar harmonic Maass forms.  We assume throughout that $
k\in\N$.
\begin{proposition}\label{prop:ellexp}
\noindent

\noindent
\begin{enumerate}[leftmargin=*]
\item[\rm(1)]
Let $\mathfrak{z}\in\mathbb{H}$ and assume that $F$ satisfies $\Delta_{2-2k} (F) = 0$ and that there exists $n_0\in \N$ such that $r_{\mathfrak{z}}^{n_0}(z) F(z)$ is bounded in some neighborhood $\mathcal{N}$ around $\mathfrak{z}$. Then there exist $a_n, b_n\in\C$, such that, for $z\in\mathcal{N}$,
\begin{multline}\label{elliptic}
F(z)=(z-\overline{\mathfrak{z}})^{2k-2}\sum_{n\geq-n_0} a_n X_{\mathfrak{z}}^n(z) - (z-\overline{\mathfrak{z}})^{2k-2}\sum_{0\leq n\leq n_0} b_n\beta\left(1-r_{\mathfrak{z}}^2(z);2k-1,-n\right) X_{\mathfrak{z}}^n(z)\\
+ (z-\overline{\mathfrak{z}})^{2k-2}\sum_{n\leq -1} b_n\beta\left(r_{\mathfrak{z}}^2(z);-n,2k-1\right) X_{\mathfrak{z}}^n(z),
\end{multline}
where $r_{\mathfrak{z}}(z):=|X_{\mathfrak{z}}(z)|$  and $\beta(y;a,b):=\int_{0}^{y} t^{a-1}(1-t)^{b-1}dt$ is the incomplete beta function.
\item[\rm(2)]  If $F\in \mathcal{H}_{2-2k}(N)$, then, for every $\mathfrak{z}\in\H$, the sums in \eqref{elliptic} only run over those $n$ which satisfy $n\equiv k-1\pmod{\omega_{\mathfrak{z}}}$.  Furthermore, if $F\in \mathcal{H}_{2-2k}^{\operatorname{cusp}}(N)$, then the second sum in \eqref{elliptic} vanishes.
\end{enumerate}
\end{proposition}
\begin{proof}
(1)  The claim follows precisely as in work of Hejhal, who computed the parabolic expansions of eigenfunctions under a differential operator closely related to the hyperbolic Laplacian in Proposition 4.3 of \cite{Hejhal}. 

\noindent

\noindent
(2)  The stabilizer group $\Gamma_\mathfrak{z}\subseteq \Gamma_0(N)$ of $\mathfrak{z}$ is cyclic and we denote by $E$ one of its generators.  In (2a.15) of \cite{Pe1}, Petersson showed that
$$
X_{\mathfrak{z}}(Ez)= e^{\frac{2\pi i}{\omega_{\mathfrak{z}}}} X_{\mathfrak{z}}(z).
$$
In particular, $r_{\mathfrak{z}}(z)
$ is invariant under $\Gamma_{\mathfrak{z}}$ and modularity of $g$ together with uniqueness of expansions in $e^{in\theta}$ implies that, for $c_n=a_n$ or $c_n=b_n$,
\begin{equation}\label{eqn:cn}
c_n(z-\overline{\mathfrak{z}})^{2k-2}=c_n e^{\frac{2\pi i n}{\omega_{\mathfrak{z}}}} j(E,z)^{2k-2}(Ez-\overline{\mathfrak{z}})^{2k-2}.
\end{equation}
Moreover, using $E\mathfrak{z}=\mathfrak{z}$, we have
$$
j(E,z)^{2k-2}(Ez-\overline{\mathfrak{z}})^{2k-2}= j\left(E,\mathfrak{z}\right)^{2k-2}\left(z-\overline{\mathfrak{z}}\right)^{2k-2}.
$$
Then, by (26) of \cite{Pe2}, we have $j\left(M,\mathfrak{z}\right)=e^{-\frac{\pi i}{\omega_{\mathfrak{z}}}}$, and thus
$$
c_n=c_n e^{\frac{2\pi i \left(n+1-k\right)}{\omega_{\mathfrak{z}}}}.
$$
Hence \eqref{eqn:cn} holds if and only if $c_n=0$ or $n\equiv k-1 \pmod{\omega_{\mathfrak{z}}}$, yielding the first statement in (2).

To conclude the second statement in (2), we apply $\xi_{2-2k}$ to \eqref{elliptic}.  Since
$$
\xi_{2-2k}(F(z))=-\left(4\mathfrak{z}_2\right)^{2k-1}(z-\overline{\mathfrak{z}})^{-2k} \sum_{n\leq n_0}\overline{b_n} X_{\mathfrak{z}}^{-n-1}(z)
$$
is a cusp form, we require that $b_n=0$ for $n\geq 0$.

\end{proof}

\section{Weight zero polar harmonic Maass Poincar\'{e} series}\label{sec:analcont}
In this section, we define a family of Poincar\'e series 
$\mathcal{P}_{N,s}$ via the Hecke trick and analytically continue them to $s=0$.  We follow an argument of Petersson \cite{Pe},
 who analytically continued certain cuspidal elliptic Poincar\'e series.  However, technical difficulties arise because the Poincar\'e series $\mathcal{P}_{N,s}$ have poles in $\mathbb{H}$.  We show in Lemma \ref{lem:xiPsi} that the analytic continuations $z\mapsto y\Psi_{2,N}(\mathfrak{z},z)$ of $\mathcal{P}_{N,s}$ to $s=0$ are elements of $\H_{0}^{\operatorname{cusp}}(N)$.  Applying certain differential operators in the $\mathfrak{z}$ variable, 
we construct a family of functions $
z\mapsto \mathcal{Y}_{0,m,N}(\mathfrak{z},z)
\in\H_{0}^{\operatorname{cusp}}(N)$ with 
principal parts $X_{\mathfrak{z}}^m(z)$ for $m\in-\N$.  
In Proposition \ref{prop:basis} 
we then prove that these functions, together with constant functions and the harmonic Maass form Poincar\'e series $\mathcal{P}_{0,n,N}^{\varrho}$ (with $n<0$), generate $\mathcal{H}_{0}^{\operatorname{cusp}}(N)$.

\subsection{Construction of the Poincar\'e series and their analytic continuations}
Define
\begin{align*}
\mathcal{P}_{N,s}(\mathfrak{z},z)&:=\sum_{M\in \Gamma_0(N)} \frac{\varphi_s(M\mathfrak{z},z)}{j(M,\mathfrak{z})^2|j(M,\mathfrak{z})|^{2s}}\qquad\text{ with }\\
\varphi_s(\mathfrak{z},z)&:=y^{1+s}(\mathfrak{z}-z)^{-1}(\mathfrak{z}-\overline{z})^{-1}|\mathfrak{z}-\overline{z}|^{-2s}.
\end{align*}
The goal is to analytically continue $\mathcal{P}_{N,s}$ to $s=0$ and show that this continuation is $\Gamma_0(N)$-invariant, as a function of $z$.  Note that the analytic continuation $\mathcal{P}_{N,0}$, if it exists, is the weight $2$ analogue (as a function of $\mathfrak{z}$) of Petersson's elliptic Poincar\'e series
\[
\Psi_{2k,N,\nu_{2k}}(\mathfrak{z}, z):=\sum_{M\in\Gamma_0(N)}\nu_{2k}(M)^{-1}j(M, \mathfrak{z})^{-2k}\left(M\mathfrak{z}-\overline{z}\right)^{1-2k}(M\mathfrak{z}-z)^{-1},
\]
where $\nu_{2k}$ is any multiplier system on $\Gamma_0(N)$ which is consistent with weight $2k$ modularity.  This function converges absolutely uniformly for $k>1$ with $\mathfrak{z},z$ in compact sets in which $M\mathfrak{z}\neq z$ for any $M\in\Gamma_0(N)$
(see Sections 1 and 2 of \cite{PeBereich}).  In particular, the absolute convergence of $\mathcal{P}_{N,s}$ follows from Petersson's work  by majorizing by the absolute values in $\Psi_{2+\sigma,
N,\nu_{2+\sigma}}$ for $\sigma=\mathrm{Re}(s)>0$.

We next consider modularity properties in the region of absolute convergence. A direct calculation shows that for, $L\in\Gamma_0 (N)$,
$$
\mathcal{P}_{N,s}(\mathfrak{z},Lz)=\mathcal{P}_{N,s}(\mathfrak{z}, z).
$$
In particular,  the analytic continuation to $s=0$, if it exists, satisfies
\[
\mathcal{P}_{N,0}(\mathfrak{z}, Lz)=\mathcal{P}_{N,0}(\mathfrak{z}, z),
\]
and hence is $\Gamma_0(N)$-invariant.  In order to study further properties of the resulting function, the goal of this section is to analytically continue $\mathcal{P}_{N,s}$ to $s\in\C$ with $\sigma >-1/4$.  To explicitly state the result, we require the Riemann zeta function $\zeta$, the Euler totient function  $\phi$, and the standard \textit{Kloosterman sums}
\begin{displaymath}\label{eqn:Kloosterman}
K(m,n;c):=\sum\limits_{\substack{a,d\pmod{c}
 \\ ad \equiv 1\pmod{c}}}e^{\frac{2\pi i}{c}\left(na+md\right)}.
\end{displaymath}
\begin{theorem}\label{thm:analcont}
The function
\begin{multline*}
2\sum_{n\in\Z} \varphi_s(\mathfrak{z}+n,z)+2\sum_{\substack{M=\left(\begin{smallmatrix}a&b\\c&d\end{smallmatrix}\right)\in \Gamma_0(N)\\ c\geq 1}} \frac{\varphi_s
(M\mathfrak{z},z)-\varphi_s\left(\frac{a}{c},z\right)
}{j(M,\mathfrak{z})^2|j(M,\mathfrak{z})|^{2s}}\\
+2\sum_{n\in\Z}\int_{\R}\varphi_s(t,z)e^{-2\pi i nt}dt\sum_{\substack{m\in\Z\\(m,n)\neq (0,0)}}\int_{\R} \left(\mathfrak{z}+w\right)^{-2-s}\left(\overline{\mathfrak{z}}+w\right)^{-s}e^{-2\pi i mw} dw \sum_{\substack{c\geq 1\\ N|c}} \frac{K(m,n;c)}{ c^{2+2s}}\\
-
\frac{2\sqrt{\pi}}{1+s}\frac{\Gamma\left(\frac{1}{2}+s\right)}{\Gamma\left(1+s\right)}\mathfrak{z}_2^{-1-2s}\cdot s\frac{\zeta(2s+1)}{\zeta(2s+2)}
\frac{\phi(N)}{N^{2+2s}}\prod_{p|N}\frac{1}{1-p^{-2-2s}}\int_{\R}\varphi_s\left(t,z\right)dt
\end{multline*}
provides the analytic continuation of $\mathcal{P}_{N,s}$ to $\sigma>-1/4$ for every $\mathfrak{z},z\in\H$ such that $M\mathfrak{z}=z$ has no solutions in $\Gamma_0(N)$.
\end{theorem}

We begin the proof of Theorem \ref{thm:analcont} by rewriting
$\mathcal{P}_{N,s}$ for $\sigma >0$ as
\begin{multline}\label{eqn:analcont}
2\sum_{n\in\Z} \varphi_s
(\mathfrak{z}+n,z)
+
2
\sum_{\substack{M=\left(\begin{smallmatrix}a&b\\c&d\end{smallmatrix}\right)\in \Gamma_0(N)\\
c\geq 1
}} \frac{\varphi_s
(M\mathfrak{z},z)
-\varphi_s
\left(\frac{a}{c},z\right)
}{j(M,\mathfrak{z})^2|j(M,\mathfrak{z})|^{2s}}\\
+
2
\sum_{\substack{M=\left(\begin{smallmatrix}a&b\\c&d\end{smallmatrix}\right)\in \Gamma_0(N)\\
c\geq 1
}}\frac{\varphi_s
\left(\frac{a}{c},z\right)
}{j(M,\mathfrak{z})^2|j(M,\mathfrak{z})|^{2s}}=:\sum_1 + \sum_2 + \sum_3.
\end{multline}

We break the proof of Theorem \ref{thm:analcont} into Lemmas \ref{lem:sum3}, \ref{lem:sum2}, and \ref{lem:sum1}, in which we obtain the analytic continuation of $\sum_3$, $\sum_2$, and $\sum_1$, respectively.
We show that $\sum_1$ and $\sum_2$ converge absolutely locally uniformly in $s$ if $\sigma>-1/2
$ for any $\mathfrak{z},z$ for which $M\mathfrak{z}=z$ has no solution $M\in\Gamma_0(N)$, which we assume throughout.  Furthermore, we claim that $\sum_3$ converges absolutely locally uniformly for $\sigma>0$ and has an analytic continuation
via its Fourier expansion to $s$ with $\sigma>-1/4$.  To validate reordering, we note that since the overall expression is absolutely locally uniformly convergent for $\sigma>0$, $\sum_3$ is absolutely locally uniformly convergent if
both
$\sum_1$ and $\sum_2$
converge absolutely locally uniformly.  We prove this convergence for $\sum_2$ in Lemma \ref{lem:sum2} and for $\sum_1$ in Lemma \ref{lem:sum1}.

\subsection{Analytic continuation of $\sum_3$}\label{subsec:sigma3}
We begin by analytically continuing $\sum_3$ in \eqref{eqn:analcont}.
\begin{lemma}\label{lem:sum3}
The function
\begin{multline}\label{eqn:rewritemainfinal}
2\sum_{n\in\Z}\int_{\R}\varphi_s
(t,z)
e^{-2\pi i nt}dt\sum_{\substack{m\in\Z\\(m,n)\neq (0,0)}}\int_{\R} \left(\mathfrak{z}+w\right)^{-2-s}\left(\overline{\mathfrak{z}}+w\right)^{-s}e^{-2\pi i mw} dw \sum_{\substack{c\geq 1\\ N|c}} \frac{K(m,n;c)}{c^{2+2s}}\\
-\frac{2\sqrt{\pi}}{1+s}\frac{\Gamma\left(\frac{1}{2}+s\right)}{\Gamma\left(1+s\right)}\mathfrak{z}_2^{-1-2s}\cdot s\frac{\zeta(2s+1)}{\zeta(2s+2)}\frac{
\phi(N)}{N^{2+2s}}\prod_{p|N}\left({1-p^{-2-2s
}}\right)^{-1}\int_{\R}\varphi_s\left(t,z\right)dt
\end{multline}
provides an analytic continuation of ${\sum}_3$ to $\sigma>-1/4$.
\end{lemma}
\begin{proof}
For $\sigma>0$, we have, using Poisson summation twice,
\begin{multline}\label{eqn:Poisson}
\sum_3=2\sum_{\substack{c\geq 1\\ N|c}}c^{-2-2s}\sum_{n\in\Z}\; \sum_{\substack{ a,d\pmod{c} \\ad\equiv 1\pmod{c}}}\int_{\R} \varphi_s
\left(\frac{a}{c}+t,z\right)
e^{-2\pi in t}dt\\
\times\sum_{m\in\Z}\int_{\R} \left(\mathfrak{z}+\frac{d}{c}+w\right)^{-2-s}\left(\overline{\mathfrak{z}}+\frac{d}{c} +w\right)^{-s}e^{-2\pi imw} dw.
\end{multline}
\indent
We next rewrite the right-hand side of \eqref{eqn:Poisson}.  Shifting $t\mapsto t -a/c$, the integral over $t$ becomes
\begin{equation}\label{eqn:ashift}
e^{\frac{2\pi i na}{c}} \int_{\R}\varphi_s(t,z) e^{-2\pi int}dt= e^{\frac{2\pi i na}{c}} y^{1+s} \int_{\R}
\left(t-z\right)^{-1-s}\left(t-\overline{z}\right)^{-1-s}  e^{-2\pi int}dt.
\end{equation}
Similarly, letting $w\mapsto w-d/c$, the integral over $w$ equals
\begin{equation}\label{eqn:dshift}
e^{\frac{2\pi i md}{c}}\int_{\R} \left(\mathfrak{z}+
w\right)^{-2-s}\left(\overline{\mathfrak{z}}+
w\right)^{-s}e^{-2\pi imw} dw.
\end{equation}
Thus one formally obtains
\begin{equation}\label{rewritemain}
\sum_3=2\sum_{n\in\Z}\int_{\R}\varphi_s
(t,z)
e^{-2\pi i nt}dt\sum_{m\in\Z}\int_{\R} \left(\mathfrak{z}+w\right)^{-2-s}\left(\overline{\mathfrak{z}}+w\right)^{-s
}e^{-2\pi i mw} dw \sum_{\substack{c\geq 1\\ N|c}} \frac{K(m,n;c)}{c^{2+2s}}.
\end{equation}

To validate \eqref{rewritemain}, one needs to verify that the triple sum converges absolutely for $\sigma>0$.
For this, we bound the Kloosterman sums trivially, to estimate the sum over $c$ against
$$
\sum_{c\geq 1} c^{-2-2\sigma}\phi(c)\ll \sum_{c\geq 1} c^{-1-2\sigma}
<\infty.
$$
\indent
It remains to show that the double sum over $n$ and $m$ converges absolutely.  Since the integrands in \eqref{eqn:ashift} and \eqref{eqn:dshift} are analytic in the integration variable for $|\im(t)|<y$ and $|\im(w)|<\mathfrak{z}_2$ respectively, we may shift the path of integration to $\im(t)=-\sgn(n)\alpha$ and $\im(w)=-\sgn(m)\beta$, respectively for any $0<\alpha<y$ and $0<\beta<\mathfrak{z}_2$.  A straightforward change of variables then shows that for any $-1/2<\sigma_0<\sigma$, the absolute value of \eqref{eqn:ashift} may be bounded against
\begin{multline}\label{eqn:intphibnd}
e^{-2\pi|n|\alpha}y^{1+\sigma} \int_{\R} |t-i(y+\sgn(n)\alpha)|^{-1-\sigma}|t+i(y-\sgn(n)\alpha)|^{
-1-\sigma
} dt
 \\
\leq e^{-2\pi|n|\alpha}\int_\R |t|^{-1-\sigma_0}  \left(\frac{t^2}{y^2}+1\right)^{
-\frac{1+\sigma_0}{2}
}dt\ll_{y,\sigma_0}e^{-2\pi |n|\alpha},
\end{multline}
while
\begin{equation}\label{eqn:jintbnd}
\left|\int_{\R} \left(\mathfrak{z}+w\right)^{-2-s}\left(\overline{\mathfrak{z}}+w\right)^{-s}e^{-2\pi i mw} dw\right|
\ll_{\mathfrak{z}_2,\sigma_0}e^{-2\pi|m|\beta}.
\end{equation}
This validates \eqref{rewritemain}.

We next split the $m=n=0$ term in \eqref{rewritemain} off and show that the remaining terms converge absolutely locally uniformly in $s$ for $-1/4<\sigma_0<\sigma$.
For this we require a well-known result of Weil, which implies that for, any $\varepsilon >0$,
\begin{equation}\label{eqn:Kbnd}
|K(m,n;c)|\leq 
\tau(c)
 c^{\frac{1}{2}}(m,n,c)^{\frac{1}{2}}\ll c^{\frac{1}{2}+\varepsilon}(m,n,c)^{\frac{1}{2}},
\end{equation}
where $\tau(c)
$ is the number of divisors of $c$.   Combining \eqref{eqn:Kbnd} with \eqref{eqn:intphibnd} and \eqref{eqn:jintbnd}, the terms in \eqref{rewritemain} with $(m,n)\neq (0,0)$ may be bounded against
\[
\ll_{\alpha,\beta,\sigma_0,\mathfrak{z}_2,y} \sum_{c\geq 1}c^{-\frac{3}{2}+\varepsilon-2\sigma_0} \sum_{n\in\Z}e^{-2\pi|n|\alpha}\sum_{m\in\Z} \left(\sqrt{|m|}+\sqrt{|n|}\right)e^{-2\pi|m|
\beta
}.
\]
Hence all sums converge absolutely uniformly in $s$ for $-1/4<\sigma_0<\sigma$.

For $m=n=0$, we use (52a) of \cite{Pe4}
(see (40a) of \cite{Pe4} for the definition of $A_0$)
to evaluate
\begin{equation}\label{eqn:Petint}
\int_{\R} (\mathfrak{z}+w)^{-2-s}(\overline{\mathfrak{z}}+w)^{-s}dw =
-\frac{\sqrt{\pi}}{1+s}\frac{\Gamma\left(\frac{1}{2}+s\right)}{\Gamma\left(1+s
\right)} s \mathfrak{z}_2^{-1-2s}.
\end{equation}
Moreover the sum over $c$ equals in this case $F(N, 2+2s)$, where
\begin{equation}\label{eqn:Fdef}
F(N,s):=\sum\limits_{\substack{n\geq 1\\ N|n}} \frac{\phi(n)}{n^s}=\frac{\phi(N)}{N^s}\sum_{n\geq 1}\frac{\phi(Nn)}{\phi(N)}n^{-s}.
\end{equation}
Using that $\phi(Nn)/\phi(N)$ is multiplicative and comparing Euler factors on both sides gives that
\begin{equation}\label{evaluateF}
F(N,s)=\frac{\phi(N)}{N^s}\prod_{p|N}\left(1-p^{-s}\right)^{-1}
\frac{\zeta(s-1)}{\zeta(s)}.
\end{equation}
Thus
\begin{equation}\label{sumcN}
\sum_{\substack{c\geq 1\\ N|c}} \frac{\phi(c)}{c^{2+2s}}=\frac{\phi(N)}{N^{2s+2}}\prod_{p|N}\left({1-p^{-2-2s}}\right)^{-1}\frac{\zeta(2s+1)}{\zeta(2s+2)}.
\end{equation}
Plugging \eqref{eqn:Petint} and \eqref{sumcN}
into
 \eqref{eqn:Poisson} for the $m=n=0$ term,
 we thus obtain that $\sum_3$ equals \eqref{eqn:rewritemainfinal}.

It remains to show that the $m=n=0$ term is indeed analytic in $s$ for $\sigma>-1/2$.  Since $\zeta(2s+2)$ does not vanish for $\sigma>-1/2$ and $\zeta(2s+1)$ only has a simple pole at $s=0$, the function $s\frac{\zeta(2s+1)}{\zeta(2s+2)}$ is analytic for $\sigma>-1/2$. The finite factor in \eqref{sumcN} is clearly analytic away from a pole at $s=-1$ and hence in particular analytic for $\sigma>-1/2$.  Similarly, the ratio of the gamma factors in \eqref{eqn:Petint} is analytic for $s$ if $\sigma>-1/2$ and $\mathfrak{z}_2^{-1-2s}$ is analytic for $s\in\C$.  It thus remains to show that $\int_{\R}\varphi_s
(t,z) dt$ is analytic in $s$.  Since $s\mapsto \varphi_{s}(t,z)$ is clearly analytic, it suffices to bound the integrand locally uniformly.  For this, we shift $t\mapsto ty+\mathrm{Re}(z)$, to obtain
$$
\int_{\mathbb{R}} \left| \varphi_s
\left( t,z\right)
 \right| dt = y^{1+\sigma}\int_{\R}|t-{z}|^{-2-2\sigma}dt=y^{-\sigma}\int_\R\left(t^2+1\right)^{-1-\sigma}dt.
$$
Assuming $-1/2<\sigma_0<\sigma<\sigma_1$ this is less than
\[
\left(y^{-\sigma_0}+y^{-\sigma_1}\right) \int_\R \left(t^2+1\right)^{-1-\sigma_0}dt.
\]
This verifies that the last term in \eqref{eqn:rewritemainfinal} is analytic for $-1/2<\sigma_0<\sigma<\sigma_1$, finishing the proof.
\end{proof}

\subsection{Analytic continuation of $\sum_2$}
In this subsection we show that $\sum_2$ converges absolutely uniformly inside the
rectangle $\mathcal{R}$ defined by $|\im(s)|\leq R$ and $-1/2<\sigma_0\leq \sigma\leq \sigma_1$.
\begin{lemma}\label{lem:sum2}
If $M\mathfrak{z}=z$ has no solution $M\in\Gamma_0(N)$, then the series $\sum_2$ converges absolutely uniformly in $\mathcal{R}$.
\end{lemma}
Before proving Lemma \ref{lem:sum2}, we show a technical lemma which proves useful later.
\begin{lemma}\label{lem:Wbnd}
Let $\mathcal{R}_0$ be the rectangle defined by $|\im(s)|\leq R$ and $-
1/2< \sigma_0\leq \sigma\leq \sigma_1$.  Then, for every $|W|<1/2$, we have
\begin{align}\label{eqn:Wbnd1}
\left| |W+1|^{-2s}-1\right| &\ll_{\mathcal{R}_0} |W|,\\
\label{eqn:Wbnd2}
\left|(W+1)^{-2s}\right| &\ll_{\mathcal{R}_0} 1.
\end{align}
\end{lemma}
\begin{proof}
\noindent
To prove \eqref{eqn:Wbnd1}, we  write $W= re^{i\theta}$ and let
$$
f_{\theta}(r)=f_{\theta, s}(r):=\left|re^{i\theta}+1\right|^{-2s} = \left(1+2r\cos(\theta)+r^2
\right)^{-s}.
$$
Since $r< 1/2$, Taylor's Theorem yields that
$$
f_{\theta}(r)=1+ f_{\theta}'(c) r
$$
for some $0<c<r< 1/2$.  But for $0<c<1/2$ and $\sigma>-1$, we have
\[
\left|f_{\theta}'(c)\right|=
|s|
\left|1+2c\cos(\theta) + c^2\right|^{
-\sigma
-1}\left|2\cos(\theta)+2c\right| \leq
3
|s|2^{
2+2\sigma}.
\]
Since inside the rectangle $\mathcal{R}_{0}$, $|s|$ and $\sigma$ are bounded
from above, we can conclude \eqref{eqn:Wbnd1}

To obtain \eqref{eqn:Wbnd2}, we note that the above proof of \eqref{eqn:Wbnd1} implies that
\[
\left\lvert (W+1)^{-2s}\right\rvert =f_{\theta, \sigma}(r)\ll_{\mathcal{R}_0} 1.
\]
\end{proof}

\indent
We next prove Lemma \ref{lem:sum2}.
\begin{proof}[Proof of Lemma \ref{lem:sum2}]

In order to show absolute locally uniform convergence of $\sum_2$, we rewrite $M$ as $T^nM$ with $n\in\Z$ and $M=\left(\begin{smallmatrix} a&b\\c&d\end{smallmatrix}\right)\in \Gamma_{\infty}\backslash \Gamma_0(N)$ such that $|\frac{a}{c}|\leq\frac{1}{2}$.  Abbreviating $w:=\frac{a}{c}+n-\overline{z}$ and $M^*\mathfrak{z}:=M\mathfrak{z}-\frac{a}{c}=-\frac{1}{c(c\mathfrak{z}+d)}$, the terms in the series $\sum_2$ equal 
\begin{multline}\label{eqn:Mwsplit}
\frac{y^{1+s}}{j(M,\mathfrak{z})^2|j(M,\mathfrak{z})|^{2s}\left| M^* \mathfrak{z} +w\right|^{2s}}\left(\frac{1}{(M^*\mathfrak{z} + w) \left( M^* \mathfrak{z} +\overline{w}\right) } - \frac{1}{|w|^2 } \right)\\
 + \frac{y^{1+s}}{j(M,\mathfrak{z})^2|j(M,\mathfrak{z})|^{2s}|w|^2}
 \left( \frac{1}{\left| M^*\mathfrak{z} + w \right|^{2s}} - \frac{1}{|w|^{2s}}\right).
\end{multline}

We next determine the asymptotic growth of \eqref{eqn:Mwsplit} in $|w|$ and $|M^*\mathfrak{z}|$, with constants only depending on $\mathcal{R}$.  For this, we rewrite the first term in \eqref{eqn:Mwsplit} as
\begin{equation}\label{eqn:Mwsplit2}
-\frac{y^{1+s}M^*\mathfrak{z}\left( M^*\mathfrak{z}+2\re(w)\right)}{j(M,\mathfrak{z})^2|j(M,\mathfrak{z})|^{2s}|w|^{2s+4}\left|\frac{M^*\mathfrak{z}}{w}+1\right|^{2s}\left(\frac{M^*\mathfrak{z}}{w}+1\right)\left(\frac{M^*\mathfrak{z}}{\overline{w}}+1\right)}.
\end{equation}
Noting that $|M^*\mathfrak{z}|=\frac{1}{|c|\cdot |c\mathfrak{z}+d|}\leq \frac{1}{\mathfrak{z}_2}$, we estimate
\begin{equation}\label{eqn:numer}
M^*\mathfrak{z}+
2\re(w)
\ll \frac{1}{\mathfrak{z}_2}+|w|.
\end{equation}

We next rewrite the second term in \eqref{eqn:Mwsplit} as
\begin{equation}\label{eqn:Mwsplit3}
\frac{y^{1+s}}{j(M,\mathfrak{z})^2|j(M,\mathfrak{z})|^{2s}|w|^{2s+2}}\left(\left|\frac{M^*\mathfrak{z}}{w}+1\right|^{-2s}-1\right).
\end{equation}
\indent
In order to bound \eqref{eqn:Mwsplit3}, we split the range on $n$
and apply Lemma \ref{lem:Wbnd} for all $n$ with $|n|$ sufficiently large.  In particular,
one can show that if $|n|\geq |z|+1/2+2/\mathfrak{z}_2$, then \eqref{eqn:Wbnd1}
implies that \eqref{eqn:Mwsplit3} can be bounded against
\[
\ll_{\mathcal{R}} \frac{y^{1+\sigma}}{|j(M,\mathfrak{z})|^{2+2\sigma}}\left\lvert M^\ast \mathfrak{z}\right\rvert |w|^{-2\sigma-3}\leq \frac{y^{1+\sigma}}{|j(M,\mathfrak{z})|^{2+2\sigma}}\left\lvert M^\ast \mathfrak{z}\right\rvert\left(|n|-|z|-\frac12\right)^{-2\sigma-3}.
\]
Moreover, for these $n$, \eqref{eqn:numer} can be bounded against $3/2\cdot |w|$.  We then
use \eqref{eqn:Wbnd2}, once with $s$ and twice with $s=1/2$, estimating \eqref{eqn:Mwsplit2} against
\[
\ll\frac{y^{1+\sigma}}{|j(M,\mathfrak{z})|^{2+2\sigma}}\left\lvert M^\ast \mathfrak{z}\right\rvert |w|^{-2\sigma-3}\ll
\frac{y^{1+\sigma}}{|j(M,\mathfrak{z})|^{2+2\sigma}}\left\lvert M^\ast \mathfrak{z}\right\rvert \left(|n|-|z|-\frac12\right)^{-2\sigma-3}.
\]
Using that $|M^*\mathfrak{z}|=|c|^{-1}|j(M,\mathfrak{z})|^{-1}$, the contribution
to $\sum_2$ from $|n|\geq |z| + 1/2+2/\mathfrak{z}_2$ may hence be bounded by
\begin{multline}\label{eqn:nbig}
\ll_{\mathcal{R}}y^{1+\sigma}\sum_{\substack{M\in \Gamma_\infty\backslash \Gamma_0 (N)\\ c\geq 1}} |c|^{-1}|j(M,\mathfrak{z})|^{-3-2\sigma}\sum_{n\geq 1} n^{-3-2\sigma} \\
\leq  \left(y^{1+\sigma_0}+y^{1+\sigma_1}\right)\zeta\left(3+2\sigma_0\right)\sum_{\substack{M\in\Gamma_\infty\backslash \Gamma_0 (N)\\ c\geq 1}} |j(M,\mathfrak{z})|^{-3-2\sigma}.
\end{multline}
The sum on $M$ is
half of
the termwise absolute value of the weight $3+2\sigma$ Eisenstein series without its constant term, and hence converges absolutely uniformly in $\mathfrak{z}$ and $\sigma>\sigma_0>-1/2$.

It remains to bound the terms of $\sum_2$ with $|n|\leq |z|+1/2+2/\mathfrak{z}_2$.  For these, we have
\begin{equation}\label{eqn:wbnd}
y\leq |w|\leq |z|+\frac12+|n|\ll |z|+1+\frac{1}{\mathfrak{z}_2}.
\end{equation}
In particular, if
\begin{equation}\label{eqn:Msmall}
\left|\frac{M^*\mathfrak{z}}{y}\right|\leq \frac{1}{2},
\end{equation}
then Lemma \ref{lem:Wbnd} can be applied.  Thus, by \eqref{eqn:Wbnd2},
\eqref{eqn:numer}, and \eqref{eqn:wbnd}, the absolute value of the terms in \eqref{eqn:Mwsplit2} with $|n|\leq|z|+1/2+2/\mathfrak{z}_2$ which satisfy \eqref{eqn:Msmall} may be bounded by
$$
\ll_{\mathcal{R}} \frac{y^{1+\sigma}\cdot \left|M^*\mathfrak{z}\right|\left(\frac{1}{\mathfrak{z}_2}+|w|\right)}{|j(M,\mathfrak{z})|^{2+2\sigma}\cdot |w|^{4+2\sigma}}\frac{}{}\ll \left(|z|+1+\frac{1}{\mathfrak{z}_2}\right)\frac{y^{-3-\sigma}}{|j(M,\mathfrak{z})|^{2+2\sigma}}\left|M^*\mathfrak{z}\right|.
$$
Similarly, \eqref{eqn:Wbnd1} and \eqref{eqn:wbnd} imply that \eqref{eqn:Mwsplit3} can be estimated against
$$
\ll_{\mathcal{R}} \frac{y^{1+\sigma}\cdot \left|M^*\mathfrak{z}\right|}{|j(M,\mathfrak{z})|^{2+2\sigma}\cdot |w|^{2\sigma+3}}\frac{}{}\ll \frac{y^{-2-\sigma}}{|j(M,\mathfrak{z})|^{2+2\sigma}}
\left|M^*\mathfrak{z}\right|.
$$
Hence the sum over the absolute value of those terms in $\sum_2$ for which $|n|\leq |z|+1/2+2/\mathfrak{z}_2$ and \eqref{eqn:Msmall} is satisfied may be bounded by
\begin{multline}\label{eqn:nsmallMbig}
\ll
\left(1+\frac{1}{y}\right)
y^{-2-\sigma}\left(|z|+1+\frac1{\mathfrak{z}_2}\right)^{2}\sum_{\substack{M\in\Gamma_\infty\backslash \Gamma_0 (N)\\ c\geq 1}} |j(M,\mathfrak{z})|^{-3-2\sigma}\\
\leq \left(1+\frac{1}{y}\right)\left(y^{-2-\sigma_0}+y^{-2-\sigma_1}\right)\left(|z|+1+\frac1{\mathfrak{z}_2}\right)^{2}\sum_{\substack{M\in\Gamma_\infty\backslash \Gamma_0 (N)\\ c\geq 1}} |j(M,\mathfrak{z})|^{-3-2\sigma}.
\end{multline}
Again, the sum over $M$ is a majorant for the weight $3+2\sigma$ Eisenstein series minus its constant term, and hence uniformly converges for $\mathfrak{z}\in\H$ and $\sigma>\sigma_0>-1/2$.

Finally, we consider the contribution from those terms in \eqref{eqn:analcont} with $|n|\ll |z|+1+1/\mathfrak{z}_2$ and $M$ which does not satisfy \eqref{eqn:Msmall},  denoting the set of such $(M,n)$ by $\mathcal{T}$.  We naively bound the contributions of $\mathcal{T}$ to \eqref{eqn:analcont}, using the original splitting
from the definition of $\sum_2$ instead of the splitting from \eqref{eqn:Mwsplit}. If $M$ does not satisfy \eqref{eqn:Msmall}, then $c^2\mathfrak{z}_2\leq |c|\cdot |c\mathfrak{z}+d|< 2y^{-1}$
and hence
\begin{equation}\label{eqn:cbnd}
|c|<\sqrt{2}\left(y\mathfrak{z}_2\right)^{-\frac{1}{2}}.
\end{equation}
For each $c$ satisfying \eqref{eqn:cbnd}, if \eqref{eqn:Msmall} is not satisfied, then $
|c\mathfrak{z}_1+d|
<|c\mathfrak{z}+d| <2(yc)^{-1}$, and hence
\begin{equation}\label{eqn:dbnd}
|d|<\left|c
\mathfrak{z}_1
\right|+ 2(yc)^{-1}\leq \left|c
\mathfrak{z}_1
\right| + 2y^{-1}< |c \mathfrak{z}|+2y^{-1}.
\end{equation}
We conclude that
\rm
$\mathcal{T}$ is finite, with $\#\mathcal{T}$ bounded by a constant only depending on $z$ and $\mathfrak{z}$. We moreover bound
\[
\frac{y^{1+s}}{j(M, \mathfrak{z})^{2}|j(M, \mathfrak{z})|^{2s}}\ll y^{1+\sigma} \mathfrak{z}_2^{-2-2\sigma}.
\]

Thus the contribution from elements in $\mathcal{T}$ to $\sum_2$
(recalling that we are using the original splitting in \eqref{eqn:analcont})
may be estimated against
\begin{multline*}
y^{1+\sigma}\mathfrak{z}_2^{-2-2\sigma}\max_{(M,n)\in\mathcal{T}} \left|\frac{1}{(M^*\mathfrak{z} + w) \left( M^* \mathfrak{z} +\overline{w}\right) \left| M^* \mathfrak{z} +w\right|^{2s}} - \frac{1}{|w|^{2s+2}}\right| \#\mathcal{T}
\\\ll_{\mathfrak{z},z,\mathcal{R}} \max_{(M,n)\in\mathcal{T}}\left(\frac{1}{ \left| M^* \mathfrak{z} +\overline{w}\right| \left| M^* \mathfrak{z} +w\right|^{1+2\sigma}}+\frac{1}{|w|^{2\sigma+2}}\right).
\end{multline*}
Since $\im\left(M^*\mathfrak{z}+w\right)\geq y$ and $|w|\geq y$, by \eqref{eqn:wbnd}, $\left|M^*\mathfrak{z}+w\right|^{-1-2\sigma}$ and $|w|^{-2\sigma-2}$ may be bounded against $\ll_y y^{-2\sigma_0}+y^{-2\sigma_1}\ll_{\mathcal{R},y}1$. We finally note that since $\mathcal{T}$ is finite, $\max_{(M,n)\in\mathcal{T}} \left| M^* \mathfrak{z} +\overline{w}\right|^{-1}$ exists unless $M^*\mathfrak{z}+\overline{w}=0$.  However, $M^*\mathfrak{z}+\overline{w}=0$ if and only if $T^nM\mathfrak{z}=z$, which is not solvable by assumption.
This implies absolute locally uniform convergence in $\mathcal{R}$.
\end{proof}

\subsection{Analytic continuation of $\sum_1$}

We finally consider $\sum_1$.
\begin{lemma}\label{lem:sum1}
If $M\mathfrak{z}=z$ has no solution $M\in\Gamma_0(N)$, then the series $\sum_1$ converges absolutely uniformly in $\mathcal{R}$.
\end{lemma}
\begin{proof}
We have
\begin{equation*}
\left|\varphi_{s}(\mathfrak{z}+n,z)\right|=y^{1+\sigma} |n|^{-2-2\sigma}\left|1+\frac{\mathfrak{z}-z}{n}\right|^{-1}\left|1+\frac{\mathfrak{z}-\overline{z}}{n}\right|^{-1-2\sigma}.
\end{equation*}
Hence, by \eqref{eqn:Wbnd2}, the contribution over $|n|\geq 2(|\mathfrak{z}|+|z|)$ to $\sum_1$ can be estimated against

\begin{equation}\label{eqn:bign}
\ll_{\mathcal{R}}\left(y^{1+\sigma_0}+y^{1+\sigma_1}\right)
\sum_{n\geq 2(|\mathfrak{z}|+|z|)} n^{-2-2\sigma}\ll_{y,\mathcal{R}} \sum_{n\geq 1}n^{-2-2\sigma}=\zeta(2+2\sigma)\leq \zeta\left(2+2\sigma_0\right).
\end{equation}
For the terms with $|n|\leq 2(|\mathfrak{z}|+|z|)$, we obtain the estimate
$$
\ll \left(|\mathfrak{z}|+|z|+1\right)\max_{|n|\leq 2(|\mathfrak{z}|+|z|)} \left|\varphi_s(\mathfrak{z}+n,z)\right|.
$$
For each of these (finitely many) $n$, we use $|T^n\mathfrak{z}-\overline{z}|\geq \im\left(T^n\mathfrak{z}-\overline{z}\right)\geq y$ to estimate
$$
\left|\varphi_s(\mathfrak{z}+n,z)\right|=y^{1+\sigma}\left|T^n\mathfrak{z}-z\right|^{-1}\left|T^n\mathfrak{z}-\overline{z}\right|^{-1-2\sigma}\leq \left|T^n\mathfrak{z}-z\right|^{-1}\left(y^{-\sigma_0}+y^{-\sigma_1}\right).
$$
Furthermore, since $T^n\mathfrak{z}\neq z$ for all $n\in\Z$ by assumption, $\max_{|n|\leq 2\left(|\mathfrak{z}|+|z|\right)}\left|T^{n}\mathfrak{z}-z\right|^{-1}$ exists.  We may hence (uniformly in $\mathcal{R}$) bound the contribution of these terms by
\begin{equation}\label{eqn:smalln}
\left(|\mathfrak{z}|+|z|+1\right)\left(y^{\sigma_0}+y^{-\sigma_1}\right)\max_{|n|\leq 2\left(|\mathfrak{z}|+|z|\right)}\left|T^{n}\mathfrak{z}-z\right|^{-1},
\end{equation}
completing the proof.
\end{proof}

Theorem \ref{thm:analcont} is now a direct consequence of \eqref{eqn:analcont} and  Lemmas \ref{lem:sum3}, \ref{lem:sum2}, and \ref{lem:sum1}.

\section{Properties of $y\Psi_{2,N}$}\label{sec:explicit}
In this section, we explicitly compute the analytic continuation $y\Psi_{2,N}(\mathfrak{z},z):=\mathcal{P}_{N,0}(\mathfrak{z},z)$ and investigate its properties.  In particular, we show that it is modular and harmonic in both variables.

\subsection{The term $\sum_3$}
In this section, we evaluate the analytic continuation of $\sum_3$. To state the result, we let
 $c_N:=-6 N^{-1}\prod_{p|N} (1+p^{-1})^{-1}=-6/[\SL_2(\Z):\Gamma_0(N)]$.
\begin{proposition}\label{prop:sum3s}
The analytic continuation of $\sum_3$ to $s=0$  is explicitly given by
\begin{equation}\label{eqn:analcontsum3}
\frac{c_N}{\mathfrak{z}_2}-8 \pi^3  \sum_{m\geq 1} m e^{2\pi i m \mathfrak{z}} \left( \sum_{n\leq 0} \sum_{\substack{c\geq 1\\ N|c}}\frac{K \left( m,n ;c \right)}{c^2}
e^{-2\pi i nz} + \sum_{n\geq 1}\sum_{\substack{c\geq 1\\ N\mid c}}\frac{K \left( m,n ;c \right)}{c^2} e^{-2\pi i n \overline{z}}\right).
\end{equation}
This function is annihilated by $\Delta_{0,z}$ and $\Delta_{2,\mathfrak{z}}$.
\end{proposition}
\begin{remark}
The fact that \eqref{eqn:analcontsum3} is annihilated by $\Delta_{0,z}$ alternatively follows by Lemma \ref{lem:xiPsi}.
\end{remark}
\begin{proof}[Proof of Proposition \ref{prop:sum3s}]
We first evaluate the second term in \eqref{eqn:rewritemainfinal} for $s=0$.
Using the fact that $\lim_{s\to 0}s
\zeta(2s+1)=1/2
$, $\Gamma(1/2)=\sqrt{\pi}$, and $\zeta(2)=\pi^2/6$, it becomes
\begin{equation}\label{eqn:1/vterm}
\frac{c_N}{\pi \mathfrak{z}_2}\int_{\R}\varphi_0
(t,z)
dt.
\end{equation}
To compute the integral in \eqref{eqn:1/vterm}, we define more generally, for $w_1\in \H\cup -\H$, $w_2\in \H$, and $n\in\Z$
$$
g_n\left(w_1,w_2\right):=\int_\R (w_1+t)^{-1}(w_2+t)^{-1}e^{-2\pi i nt} dt.
$$
Shifting the path of integration to $-i\sgn (n-1/2)\infty$, the Residue Theorem yields
\begin{equation}\label{eqn:gval}
g_{n}\left(w_1,w_2\right)
=\begin{cases}
0 &\text{if }n\leq 0\text{ and }w_1\in \H,\\
2\pi i \left(w_2-w_1\right)^{-1}
e^{2\pi inw_1
}&\text{if }n\leq 0\text{ and } w_1
\in -\H,\\
2\pi i \left(w_2-w_1\right)^{-1} e^{2\pi i n w_2}&\text{if }n> 0\text{ and }w_1\in -\H,\\
2\pi i \left(w_2-w_1\right)^{-1} \left(e^{2\pi i n w_2}- e^{2\pi i n w_1}\right)&\text{if }n>0
\text{ and } w_2\neq w_1\in\H,\\
-4\pi^2 n e^{2\pi i n w_1}&\text{if }n>0\text{ and }w_1=w_2.
\end{cases}
\end{equation}
Thus we in particular obtain
$$
\int_{\R}\varphi_0(t,z)dt =yg_0(-z,-\overline{z})=\pi,
$$
and hence \eqref{eqn:1/vterm} equals $c_N/\mathfrak{z}_2$, giving the first term in \eqref{eqn:analcontsum3}.

Next we turn to the first term in \eqref{eqn:rewritemainfinal} with $s=0$.
To simplify the sums over $m$ and $n$, we rewrite
\begin{align*}
\int_{\R} \varphi_0
\left( t,z \right)
 e^{-2\pi i nt} dt
&= y g_n \left( - z, -\overline{z} \right),\\
\int_{\R} \frac{1}{(\mathfrak{z} + w)^2} e^{-2\pi i m w} dw &= g_m \left( \mathfrak{z}, \mathfrak{z}\right).
\end{align*}
\indent
Plugging in  \eqref{eqn:gval} yields the sum over $m>0$ in \eqref{eqn:analcontsum3}. This series converge absolutely locally uniformly on compact subsets of $\H\times\H$ due to the exponential decay in $\mathfrak{z}_2$ and $y$ in the sums over $m$ and $n$, respectively.  Hence the function is harmonic in both $\mathfrak{z}$ and $z$ because it is termwise.
\end{proof}

\subsection{The term $\sum_2$} We next consider $\sum_2$ in \eqref{eqn:analcont}.

\begin{proposition}\label{prop:sum2s}
The series $\sum_2$ with $s=0$ converges absolutely locally uniformly in $\mathfrak{z}$ and $z$ if $M\mathfrak{z}=z$ is not solvable for $M\in \Gamma_0(N)$ and is meromorphic as a function of $\mathfrak{z}$ and harmonic as a function of $z$.
\end{proposition}
\begin{proof}
After rewriting, the series $\sum_2$ with $s=0$ becomes
\begin{equation}\label{eqn:sum20}
i
\sum_{\substack{M=\left(\begin{smallmatrix}a&b\\c&d\end{smallmatrix}\right)\in\Gamma_0(N)\\
c\geq 1
}} \frac{ \frac{1}{\left(M\mathfrak{z}-z\right)\left(\frac{a}{c}-z\right)}-\frac{1}{\left(M\mathfrak{z}-\overline{z}\right)\left(\frac{a}{c}-\overline{z}\right)}}{c j(M, \mathfrak{z})^3}.
\end{equation}
Each summand is meromorphic as a function of $\mathfrak{z}$ and harmonic
 as a function of $z$.
It hence suffices to prove locally uniform convergence in $\mathfrak{z}$ and $z$ to
show that
 $\sum_2$ has the desired properties.
Since the argument for $z$ is similar, we only prove the statement for $\mathfrak{z}$.

We begin by constructing local neighborhoods around each $\mathfrak{z}\in\H$ for which $M\mathfrak{z}=z$ with $M\in \Gamma_0(N)$ does not have a solution.
Since for each $\mathfrak{z}\in\H$, $\left\{ M\mathfrak{z}\middle| M\in \Gamma_0(N)\right\}$ is a lattice in $\H$, $\delta_z(\mathfrak{z})=\delta_{z,N}(\mathfrak{z}):=\min_{M\in\Gamma_0(N)} \left|M\mathfrak{z}-z\right|$ exists.  For $\delta,V,R>0$, we then define the set
$$
\mathcal{N}_{z}(\delta,V,R):=\left\{ \mathfrak{z}\in\H\middle|\; \delta_z(\mathfrak{z})\geq \delta,\ \frac{R}{2}\leq |\mathfrak{z}|\leq 2R,\ \frac{V}{2}\leq \mathfrak{z}_2\leq 2V\right\}.
$$

We first claim that for every $\tau_0\in\H$ with $\delta_{z}(\tau_0)\neq 0$, there exists $\delta>0$ sufficiently small (depending on $\tau_0$) such that for $R=|\tau_0|$ and $V=v_0$, the set $\mathcal{N}_{z}(\delta,V,R)$ is a neighborhood of $\tau_0$.
In particular, we show that for $\varepsilon>0$ sufficiently small, $\mathcal{N}_{z}(\delta,v_0,|\tau_0|)$ contains the ball of radius $\varepsilon$ around $\mathfrak{z}=\tau_0$.  Firstly, if
 $|\mathfrak{z}-\tau_0|<\varepsilon$ for $\varepsilon>0$ sufficiently small, then
the last two conditions required for elements of
 $\mathcal{N}_{z}(\delta,V,R)$ are satisfied.
It hence remains to show that for all $M\in\Gamma_0(N)$ $|M\mathfrak{z}-z|>\delta$ if $\delta$ and $\varepsilon$ are sufficiently small.
To see this, we first note that $\beta_z(\mathfrak{z}):=\min_{M\in\Gamma_0(N)}|\tau_0-M^{-1}z|>0$ exists
and
 $\beta_z(\mathfrak{z})=0$ if and only if $\delta_{z}(\mathfrak{z})= 0$.  If $\varepsilon<\beta_z(\tau_0)$, then the triangle inequality implies that
$$
\left|M\mathfrak{z}-z\right|=\left|\frac{j(M,\mathfrak{z})}{j(M,z)}\right| \left|\mathfrak{z}-M^{-1}z\right|\geq \left|\frac{j(M,\mathfrak{z})}{j(M,z)}\right|\left( \left|\tau_0-M^{-1}z\right|-\left|\mathfrak{z}-\tau_0\right|\right)
\geq  \left|\frac{j(M,\mathfrak{z})}{j(M,z)}\right|\left(\beta_z\!\left(\tau_0\right)-\varepsilon\right).
$$

For $c=0$ this immediately gives that $\left|M\mathfrak{z}-z\right| \geq \beta_z\!\left(\tau_0\right)-\varepsilon$, and hence, for $\delta$ sufficiently small (depending on $\tau_0$ but independent of $M$), we have $\left|M\mathfrak{z}-z\right|>\delta$.  For $c\neq 0$ we rewrite
\begin{equation}\label{eqn:jtaujz}
\left|\frac{j(M,\mathfrak{z})}{j(M,z)}\right|=\frac{\left|\mathfrak{z}+\frac{d}{c}\right|}{\left|z+\frac{d}{c}\right|}.
\end{equation}
As $d/c\to\pm \infty$, \eqref{eqn:jtaujz}
converges to $1$, while for $d/c\to 0$, \eqref{eqn:jtaujz} converges to $|\mathfrak{z}/z|$.  Thus \eqref{eqn:jtaujz}
attains a minimum $\mathcal{J}_{z}(\mathfrak{z}):=\min_{M\in\Gamma_0(N)}|j(M,\mathfrak{z})/j(M,z)|>0.$  One sees directly that $\mathcal{J}_{z}(\mathfrak{z})$ is continuous as a function of $\mathfrak{z}$, and hence, for $|\mathfrak{z}-\tau_0|<\varepsilon$ satisfying $\varepsilon<\beta_z(\tau_0)$ and $\varepsilon<\mathcal{J}_{z}(\tau_0)$, we have
$$
|M\mathfrak{z}-z|\geq \left(\mathcal{J}_{z}\left(\tau_0\right)-\varepsilon\right)\left(\beta_z(\tau_0)-\varepsilon\right)>0.
$$
Choosing $\varepsilon$ and $\delta$ sufficiently small (again, depending on $\tau_0$ but independent of $M$), we conclude that $|M\mathfrak{z}-z|\geq \delta$ and hence $\mathfrak{z}\in \mathcal{N}_{z}(\delta,v_0,|\tau_0|)$.  But then $\mathcal{N}_{z}(\delta,v_0,|\tau_0|)$ contains the open ball around $\tau_0$ of radius $\varepsilon$, and is hence a neighborhood of $\tau_0$.
\rm

We next claim that the series $\sum_2$ converges uniformly in $\mathcal{N}_{z}(\delta,V,R)$.
For this, we require a uniform bound for $\sum_{\substack{M\in\Gamma_\infty\backslash \Gamma_0 (N)}} |j(M,\mathfrak{z})|^{-3-2\sigma}$.  Since this series is the termwise absolute value of the weight $3+2\sigma$ Eisenstein series, it is well-known to be smaller than a uniform constant times the value with $\mathfrak{z}=i$.
Thus \eqref{eqn:nbig} implies that the contribution to $\sum_2$ from the terms with $|n|>|z|+1/2+2/\mathfrak{z}_2$ may be bounded absolutely uniformly on any compact subset of $\H$.  Similarly, \eqref{eqn:nsmallMbig} implies a uniform estimate on compact subsets for the contribution of the terms with $|n|\leq |z|+1/2+2/\mathfrak{z}_2$ satisfying \eqref{eqn:Msmall}.

It hence remains to uniformly estimate the sum of the absolute value of those $(M,n)\in\mathcal{T}$ for $\mathfrak{z}\in \mathcal{N}_{z}(\delta,V,R)$.  We begin by bounding
$$
y\sum_{(M,n)\in\mathcal{T}}\left| \frac{1}{\left(M^*\mathfrak{z}+w\right) \left(M^*\mathfrak{z}+\overline{w}\right)}-\frac{1}{|w|^2}\right| \leq y\sum_{(M,n)\in\mathcal{T}}\frac{1}{\left|M^*\mathfrak{z}+w\right| \left|M^*\mathfrak{z}+\overline{w}\right|} + y\sum_{(M,n)\in\mathcal{T}}\frac{1}{|w|^2}.
$$
Since $|w|\geq y$ by \eqref{eqn:wbnd}, the second sum is bounded by
$$
 y\sum_{(M,n)\in\mathcal{T}}\frac{1}{|w|^2}\leq
 \frac{1}{y}
\#\mathcal{T}.
$$

For the first sum, we use the inequalities $|M^*\mathfrak{z}+w|\geq \im(M^*\mathfrak{z}+w)\geq y$ and
$$
\left|M^*\mathfrak{z}+\overline{w}\right|\geq \min_{(M,n)\in\mathcal{T}} \left|M^*\mathfrak{z}+\overline{w}\right| \geq \delta_{z}(\mathfrak{z})\geq \delta
$$
to obtain
$$
y\sum_{(M,n)\in\mathcal{T}}\frac{1}{\left|M^*\mathfrak{z}+w\right| \left|M^*\mathfrak{z}+\overline{w}\right|} \leq
 \frac{1}{ \delta}
\#\mathcal{T}.
$$
Thus
\begin{equation*}
y
\sum_{(M,n)\in\mathcal{T}}\left| \frac{1}{\left(M^*\mathfrak{z}+w\right) \left(M^*\mathfrak{z}+\overline{w}\right)}-\frac{1}{|w|^2}\right|
\leq
\# \mathcal{T}\left(\frac{1}{
\delta}+\frac{1}{y}\right)\ll_{\delta,y}
 \#\mathcal{T}.
\end{equation*}

We next note that
$$
\#\mathcal{T}\ll \left(|z|+1+\frac{1}{\mathfrak{z}_2}\right)\#\left\{ M\in \Gamma_{\infty}\backslash
\Gamma_0(N)\middle| \left|M^*\mathfrak{z}\right|> \frac{y}{2}\right\}.
$$
However, \eqref{eqn:cbnd} and \eqref{eqn:dbnd} imply that
\begin{equation}\label{eqn:Tbnd}
\#\left\{ M\in \Gamma_{\infty}\backslash\Gamma_0(N)
\middle| \left|M^*\mathfrak{z}\right|> \frac{y}{2}\right\}\ll \left(\mathfrak{z}_2y\right)^{-\frac{1}{2}}\left(\left(\mathfrak{z}_2y\right)^{-\frac{1}{2}}|\mathfrak{z}| +y^{-1}\right).
\end{equation}
Since the right-hand side of \eqref{eqn:Tbnd} is continuous in $|\mathfrak{z}|$ and $\mathfrak{z}_2$, it may be uniformly bounded on $\mathcal{N}_z(\delta,V,R)$.
This implies absolute locally uniform convergence of $\sum_2$, as desired.
\end{proof}

\subsection{The term $\sum_1$ }
We next investigate the properties of the analytic continuation of $\sum_1$ to $s=0$.
\begin{proposition}\label{prop:sum1s}
The series $\sum_1$ with $s=0$ converges locally uniformly in $\mathfrak{z}$ and $z$
for which $M\mathfrak{z}=z$ is not solvable with $M\in \Gamma_0(N)$
and is meromorphic as a function of $\mathfrak{z}$ and harmonic as a function of $z$.
\end{proposition}
\begin{proof}
For $s=0$, the series $\sum_1$ becomes
\begin{equation}\label{sum10}
2\sum_{n\in\Z}  \varphi_0(\mathfrak{z}+n,z) = -i\sum_{n\in\Z}\left(\frac{1}{\mathfrak{z}+n-z}-\frac{1}{\mathfrak{z}+n-\overline{z}}\right).
\end{equation}

Each term in (\ref{sum10}) is holomorphic as a function of $\mathfrak{z}$ and harmonic as a function of $z$.
We again only prove locally uniform convergence in $\mathfrak{z}$ and
leave the analoguous argument for $z$ to the reader.  Noting that, for
every $\mathfrak{z}\in \mathcal{N}_z(\delta,V,R)$, the inequality
$$
\max_{|n|\leq 2\left(|\mathfrak{z}|+|z|\right)}\frac{1}{\left|T^{n}\mathfrak{z}-z\right|}\leq \frac{1}{\delta}
$$
is satisfied, one obtains a uniform bound by \eqref{eqn:bign} and \eqref{eqn:smalln}.
\end{proof}

\subsection{Image under $\xi_{0,z}$}\label{sec:xi}
As mentioned in the introduction, we obtain the principal part condition from the  Riemann--Roch Theorem via a pairing of Bruinier and Funke \cite{BF}.  This pairing is between weight $2k$ cusp forms and weight $2-2k$ polar harmonic Maass forms which map to cusp forms under the operator $\xi_{2k,z}$ in the splitting \eqref{eqn:Deltasplit}.
For this reason, it is important to compute
the image of $y\Psi_{2,N}(\mathfrak{z},z)$ under $\xi_{0,z}$ and prove that it is a cusp form, as
we do in this section.

For $s,s'$ with real parts $>1$ and $\tau_0,z_0\in\H$, we additionally require the functions 
\[
P_{N,s,s'}\left(\mathfrak{z}, \tau_0, z, z_0\right):=\sum_{M\in\Gamma_0(N)}\frac1{j(M, \mathfrak{z})^2 |j(M, \tau_0)|^s (M\mathfrak{z}-\overline{z})^2|M\tau_0-\overline{z}_0|^{s^{'}}}.
\]
\indent
Petersson \cite{Pe} studied related functions on $\Gamma(N)$; 
one obtains $P_{N,s,s'}$ by taking a trace of Petersson's functions.  Section 4 of \cite{Pe} then implies that $P_{N,s,s'}(\mathfrak{z},\tau_0,z,z_0)$ has an analytic continuation to $s'=s=0$ which is independent of $\tau_0$ and $z_0$, denoted here by $\Phi_{N}(\mathfrak{z},z)$.

\begin{lemma}\label{lem:xiPsi}
We have $z\mapsto
y\Psi_{2,N}(\mathfrak{z}, z)\in \H_{0}^{\operatorname{cusp}}(N)$ with
\begin{equation}\label{eqn:xi0}
\xi_{0, z} \left(y\Psi_{2,N}(\mathfrak{z},z)\right)=\Phi_N(z,\mathfrak{z}).
\end{equation}
\end{lemma}
\begin{proof}
A direct calculation gives
\[
\frac{\partial}{\partial \overline{z}}
\left(y^{1+s}\left(M\mathfrak{z}-\overline{z}\right)^{-\left(1+s\right)}\right)
=
\frac{i}{2}(1+s)y^s
\left(M\mathfrak{z}-\overline{z}\right)^{-\left(2+s\right)}(M\mathfrak{z}-z).
\]
Thus, using locally uniform convergence in $\overline{z}$, which can be shown by an argument similar to the proofs of Propositions \ref{prop:sum3s}, \ref{prop:sum2s}, and \ref{prop:sum1s},
\begin{equation}\label{eqn:Ds}
\frac{\partial}{\partial\overline{z}}
\mathcal{P}_{N,s}(\mathfrak{z}, z)=\frac{i}{2}(1+s)y^s P_{N,2s,2s}(\mathfrak{z},\mathfrak{z},z,z).
\end{equation}

Taking the analytic continuation of both sides of \eqref{eqn:Ds} to $s=0$, we conclude that
$$
\frac{\partial}{\partial\overline{z}}
\left(y\Psi_{2,N}(\mathfrak{z}, z)\right)=
\frac{i}{2}
\Phi_N(\mathfrak{z},z).
$$
By conjugating both $P_{N,s,s'}$ and $s$ and then using (23) of \cite{Pe2} to switch the role of the variables, we conclude that

$$
\Phi_N(\mathfrak{z},z) = \overline{\Phi_N(z,\mathfrak{z})},
$$
implying \eqref{eqn:xi0}.  Finally, by Satz 2 of \cite{Pe},
$\Phi_N(z, \mathfrak{z})$ is a cusp form, giving the final claim.
\end{proof}

\section{Expansions of $y\Psi_{2,N}(\mathfrak{z},z)$ in other cusps and limiting behavior towards the cusps and the proof of Theorem \ref{thm:differentials}}\label{sec:cusps}
In this section, we determine the principal parts of $z\mapsto y\Psi_{2,N}(\mathfrak{z},z)$ and then construct a basis of $\H_{0}^{\operatorname{cusp}}(N)$ by applying differential operators to $y\Psi_{2,N}(\mathfrak{z},z)$ in the $\mathfrak{z}$ variable. For this, we show that $y\Psi_{2,N}(\mathfrak{z},z)$ is the $k=1$ analogue of $y^{2k-1}\Psi_{2k,N}(\mathfrak{z},z)$ in the sense that their principal parts are as expected if $k=1$.  Before stating the proposition, we note that the principal parts coming from the meromorphic parts of $(2iy)^{2k-1}\Psi_{2k,N}(\mathfrak{z},z)$ were computed as a special case of (50) of \cite{Pe2}; we expound further upon this analogy in Lemma \ref{lem:basis}.

\begin{proposition}\label{prop:Psigrowth}
If $\mathfrak{z}$ is an elliptic fixed point, then $z\mapsto y\Psi_{2,N}(\mathfrak{z},z)$ vanishes identically. If $\mathfrak{z}\in\H$ is not an elliptic fixed point, then, for every $M\in \Gamma_0(N)$, $y\Psi_{2,N}(\mathfrak{z},z)$ has principal part $\frac{j(M,\overline{\mathfrak{z}})}{2\mathfrak{z}_2 j(M,\mathfrak{z})}X_{M\mathfrak{z}}^{-1}(z)$ around $z=M\mathfrak{z}$ and is bounded towards all cusps.
\end{proposition}
\rm
To see the statement for an elliptic fixed point $\mathfrak{z}$, we rewrite each element in
the sum over $\Gamma_0(N)$ in the definition of $\mathcal{P}_{N,s}$ as $ME^r$ with $M\in \Gamma_0(N)/\Gamma_{\mathfrak{z}}$, with $E$ generating $\Gamma_{\mathfrak{z}}$ and $r$ running $\pmod{2\omega_{\mathfrak{z}}}$.  The sum over $r$ then becomes
$$
\sum_{r\pmod{2\omega_{\mathfrak{z}}}} j\left(E^r,\mathfrak{z}\right)^{-2}\left|j\left(E^r,\mathfrak{z}\right)\right|^{-2s} =\sum_{r\pmod{2\omega_{\mathfrak{z}}}}e^{\frac{2\pi ir}{\omega_{\mathfrak{z}}}}=
\begin{cases} 
0&\text{if }\omega_{\mathfrak{z}}\neq 1,\\ 
2&\text{if }\omega_{\mathfrak{z}}=1,
\end{cases}
$$
where we used (26) of \cite{Pe2} (note that here $a=-1$) to evaluate $j\left(E^r,\mathfrak{z}\right)^{2}=e^{-\frac{2\pi i r}{\omega_{\mathfrak{z}}}}$.  Thus $\mathcal{P}_{N,s}$
vanishes if $\mathfrak{z}$ is an elliptic fixed point, and hence its analytic continuation is simply the zero function.  

If $\mathfrak{z}$ is not an elliptic fixed point, then the possible poles of $y\Psi_{2,N}(\mathfrak{z},z)$ come from the terms in $\sum_1$ and $\sum_2$ for which $M\mathfrak{z}=z$, since $\sum_3$ converges for all $\mathfrak{z},z\in\H$. Furthermore, by determining terms of \eqref{eqn:sum20} and \eqref{sum10} which contribute to the pole, a direct calculation yields that the residue of the principal part at $z=M\mathfrak{z}$ must be $i/j(M,\mathfrak{z})^2$.  From this, one concludes that the principal part is $\frac{j(M,\overline{\mathfrak{z}})}{2\mathfrak{z}_2 j(M,\mathfrak{z})} X_{M\mathfrak{z}}^{-1}(z)$.

 In order to prove Proposition \ref{prop:Psigrowth}, it hence remains to determine the growth of $y\Psi_{2,N}(\mathfrak{z},z)$ as $z$ approaches a cusp.  This is proven in a series of lemmas.

\subsection{Cusp expansions}
In this section we rewrite the function $\mathcal{P}_{N,s}$ in order to understand the behavior if $z$ is close to a cusp $\alpha/\gamma$, with $\gamma | N$, $\gamma \neq N$, and $(\alpha,N)=1$.   Letting $L:= \left(\begin{smallmatrix} \alpha & \beta \\  \gamma & \delta \end{smallmatrix}\right)$ with  $\alpha\delta \equiv 1\pmod{N}$, it is easy to see that
$$
\mathcal{P}_{N,s}\left( \mathfrak{z}, L z\right) =\sum_{M\in L^{-1} \Gamma_0 (N)}\frac{\varphi_s\left( M\mathfrak{z},z \right)}{j \left( M, \mathfrak{z}\right)^2 \left| j \left( M, \mathfrak{z} \right)\right|^{2s}}.
$$
Hence, for $\sigma>-1/4$, the behavior of (the analytic continuation of) $\mathcal{P}_{N,s}$ as $z$ approaches the cusp $L(i\infty)$ may be determined by taking $z\to i\infty$ on (the analytic continuation of) the right-hand side.

To determine this continuation, we decompose as in \eqref{eqn:analcont} and denote the corresponding sums by $\sum_1(\alpha,\gamma), \sum_2(\alpha,\gamma)$, and $\sum_3(\alpha,\gamma).$  Note, that since $\gamma \neq N$, $\sum_1(\alpha,\gamma)$ cannot occur. Moreover $\sum_2(\alpha, \gamma)$ is treated analogously to $\sum_2$,
 and in particular converges uniformly for $-1/2<\sigma_0\leq \sigma\leq \sigma_1$ under the assumption that $M\mathfrak{z}=z$ is not solvable in $L^{-1}\Gamma_0 (N)$.
Thus, if $\sigma>-1/2$, we may directly compute the Fourier expansion of $\sum_2(\alpha,\gamma)$ for $y$ sufficiently large to determine the growth towards the cusp $\alpha/\gamma$.  We do so in the proof of Lemma \ref{lem:alpha/gammalim}.

We are thus left to consider ${\sum}_3(\alpha, \gamma)$.  A direct calculation shows that
$$
L^{-1} \Gamma_0 (N) =
\left\{
\left(\begin{matrix} a& b\\ c&d\end{matrix}\right)\in\SL_2 (\Z)\middle|
c \equiv -a\alpha \gamma\pmod{N}
\right\}.
$$

It is not hard to see that
the \begin{it}cusp width\end{it} of $\alpha/\gamma$, the minimal $\ell$ such that $\Gamma_{\infty}^{\ell}$ acts on $L^{-1}\Gamma_0(N)$ from the left, is
$\ell=\ell_{\varrho}:= \frac{N/\gamma}{\left( N/\gamma, \gamma \right)}$.  Moreover $\Gamma_\infty$ acts on $L^{-1}\Gamma_0(N)$ from the right.  Thus we obtain, as in Section \ref{subsec:sigma3},
\begin{multline}\label{eqn:sum3cuspexp}
{\sum}_3(\alpha,\gamma)= 2 \ell^{-1-s} \sum_{n\in\Z}\int_{\R} \varphi_s\left(t, \frac{z}{\ell}  \right) e^{-2\pi i nt} dt \\
\times\sum_{m\in\Z} \int_\R \left( \mathfrak{z} +w \right)^{-2-s}
\left( \overline{\mathfrak{z}} + w \right)^{-s}
 e^{-2\pi i mw} dw
 \sum_{
c\geq 1
} c^{-2-2s} K_{\alpha, \gamma} \left( m,n;c \right),
\end{multline}
where $K_{\alpha,\gamma}$ is defined in \eqref{eqn:Kagdef}.

We next prove that the analogue of \eqref{eqn:Kbnd} holds. For this, we write $K_{\alpha,\gamma}$ in terms of the classical Kloosterman sums.  To state the resulting identity, we require the natural splitting $\ell= \ell_1 \ell_2$ with $\ell_1\mid \gamma^{\infty}$ and $(\ell_2, \gamma) = 1$,
where $\ell_1\mid \gamma^{\infty}$ means that there exists $n\in\N_0$ such that $\ell_1\mid \gamma^n$.  A straightforward calculation then shows the following.
\begin{lemma}\label{rewriteKlos}
We have
$$
K_{\alpha, \gamma} \left( m,n;c\right) = \frac{1}{N_1}
e_{\ell_2} \left(- \left[\ell_1 \gamma\alpha \right]_{\ell_2} n \right)
\sum_{
r 
\pmod{N_1}} e_{N_1} \left( 
r
 \left[ \ell_2 \alpha \right]_{N_1} \frac{c}{\gamma} \right)K \left(\ell_1 \left[ \ell_2 \right]_{\ell_1 c} m, n+\frac{\ell_1 c}{N_1}
r;
 \ell_1 c \right),
$$
where  $N/\gamma = N_1 \ell_2$, and $[a]_b$ denotes the inverse of $a\pmod{b}$.
\end{lemma}
From Lemma \ref{rewriteKlos}, the analogue of \eqref{eqn:Kbnd} then follows easily and we can argue as before for the terms $(m, n)\neq (0, 0)$.

We finally rewrite the contribution from $m=n=0$ in a form which yields its analytic continuation to $s=0$.
In order to state the result, we
 let $\ell_1 \gamma=A_1A_2$, with $A_1\mid N_1^{\infty}$ and $\left(A_2,N_1\right)=1$ and denote  the M\"obius function by $\mu$.

\begin{lemma}\label{lem:53}
The
analytic continuation to $\sigma>-1/2$ of the
$m=n=0$ term in $\sum_3(\alpha,\gamma)$ exists and equals
\begin{multline}\label{eqn:mn0cuspfinal}
-\frac{2\sqrt{\pi}}{1+s}\ell^{1+s}\frac{\Gamma\left(\frac{1}{2}+s\right)}{\Gamma(1+s)} s \mathfrak{z}_2^{-1-2s} \frac{A_1}{N_1\phi(\ell_2A_1)}\frac{\zeta(2s+1)}{\zeta(2s+2)}\int_{\R}\varphi_{s}\left(t,\frac{z}{\ell}\right)dt\\
\times \sum_{g\mid \frac{N}{\gamma}} \mu(g)\frac{\phi(g\gamma \ell)}{(g
\gamma\ell
)^{2+2s}} \prod_{p\mid g\gamma\ell} \frac{1}{1-p^{-2-2s}}.
\end{multline}
\end{lemma}
\begin{proof}
A direct calculation shows that $K_{\alpha ,\gamma}(0,0;c)$ vanishes unless $(c,N)=\gamma$, in which case it equals
\begin{equation*}
K_{\alpha,\gamma}(0,0;c)=
\frac{A_1}{N_1}
\phi\left(\frac{A_2 c}{\gamma}\right).
\end{equation*}
Using \eqref{eqn:Petint} and letting $c\mapsto c\gamma$, we then easily obtain that the $m=n=0$ term in $\sum_3 (\alpha, \gamma)$  equals
\begin{equation}\label{eqn:mn0cuspgamma}
-\frac{2\ell^{-1-s}\sqrt{\pi}}{1+s}\frac{\Gamma\left(\frac{1}{2}+s\right)}{\Gamma(1+s)} s \mathfrak{z}_2^{-1-2s}\frac{A_1}{N_1}\sum_{\substack{c\geq 1\\ \left(c,
\frac{N}{\gamma}\right)=1}}\frac{ \phi\left(A_2c\right)}{(c\gamma)^{2+2s}}\int_{\R} \varphi_{s}\left(t,\frac{z}{\ell}\right)dt.
\end{equation}

We next rewrite the sum over $c$ as
$$
\sum_{\substack{c\geq 1\\  \left(c,\frac{N}{\gamma}
\right)=1}} \frac{ \phi\left(A_2c\right)}{(c\gamma)^{2+2s}} = \frac{\ell^{2+2s}}{\phi(A_1\ell_2)}\sum_{\substack{c\geq 1\\  (c,
\frac{N}{\gamma}
)=1}} \frac{ \phi\left(A_2c\right)\phi\!\left(A_1\ell_2\right)}{(c\gamma\ell)^{2+2s}}.
$$
One easily checks that $(A_2c,A_1\ell_2)=1$ and then uses the multiplicativity of $\phi$ together with $A_1A_2\ell_2=\gamma\ell$ to rewrite the right-hand side as
\begin{equation}\label{eqn:torewriteF}
\frac{\ell^{2+2s}}{\phi(A_1\ell_2)}\sum_{\substack{c\geq 1\\  \left(c,\frac{N}{\gamma}\right)=1}} \frac{ \phi(c\gamma\ell)}{(c\gamma\ell)^{2+2s}}=\frac{\ell^{2+2s}}{\phi\left(A_1\ell_2\right)}\sum_{\substack{c\geq 1\\ \gamma\ell \mid c\\ \left(\frac{c}{\gamma\ell},
\frac{N}{\gamma}
\right)=1}} \frac{\phi(c)}{c^{2+2s}}
=\frac{\ell^{2+2s}}{\phi\left(A_1\ell_2\right)}
\sum_{\substack{c\geq 1\\ \gamma\ell \mid c}}
\sum_{g\mid \left(\frac{c}{\gamma\ell},\frac{N}{\gamma}\right)} \mu(g)\frac{ \phi(c)}{c^{2+2s}}.
\end{equation}
Here the last equality holds by the well-known identity
$$
\sum_{g\mid n}\mu(g)=\begin{cases} 1 &\text{if }n=1,\\ 0&\text{if }n>1.\end{cases}
$$

Reversing the order of summation
then yields that the right-hand side of \eqref{eqn:torewriteF} equals
$$
\frac{\ell^{2+2s}}{\phi\left(A_1\ell_2\right)}\sum_{g\mid \frac{N}{\gamma}} \mu(g) F(g\gamma\ell,2+2s),
$$
where $F(N,s)$ is defined in \eqref{eqn:Fdef}.  Using \eqref{evaluateF}, \eqref{eqn:mn0cuspgamma} hence becomes \eqref{eqn:mn0cuspfinal}, completing the proof.
\end{proof}

\subsection{Behavior towards the cusps}
We are now ready to determine the growth as $z$ approaches a cusp.  For this we compute the Fourier expansion of $z\mapsto y \Psi_{2,N} \left( \mathfrak{z}, z\right)$.
\begin{lemma}\label{lem:alpha/gammalim}
For $L\in\SL_2(\Z)$, consider the cusp $L(i\infty)=\alpha/\gamma$ with $\gamma\mid N$ and $(\alpha,N)=1$, and let $v_0>0$.  Then, for $x_L+iy_L=z_L:=Lz$
satisfying $2v_0<\mathfrak{z}_2<y-1/v_0$, we have
\begin{equation}\label{eqn:alpha/gammalim}
\begin{split}
y_L\Psi_{2,N}\left(\mathfrak{z},z_L\right)=&2\pi \delta_{\frac{\alpha}{\gamma},\infty}\left(\sum_{n\geq 0} e^{-2\pi i n\mathfrak{z}} e^{2\pi i n
z}+ \sum_{n\geq 1} e^{2\pi in\mathfrak{z}} e^{-2\pi i n \overline{z}}\right)\\
&+ 4\pi^2 \sum_{c\geq 1}c^{-1} \sum_{n\geq 1} n^{-\frac{1}{2}} e^{\frac{2\pi i n z}{\ell}
} \sum_{m\geq 1} m^{\frac{1}{2}}e^{2 \pi i m \mathfrak{z}} K_{\alpha,\gamma}(m,-n;c)I_1\left(\frac{4\pi\sqrt{mn}}{c\ell}\right)\\
&+4\pi^2 \sum_{c\geq 1}c^{-1} \sum_{n\geq 1 } n^{-\frac{1}{2}} e^{-\frac{2\pi i n \overline{z
}}{\ell}} \sum_{m\geq 1} m^{\frac{1}{2}}e^{2 \pi i m \mathfrak{z}} K_{\alpha,\gamma}(m,n;c)J_1\left(\frac{4\pi\sqrt{mn}}{c\ell}\right)\\
&+\frac{c_N}{\mathfrak{z}_2}-  \frac{8 \pi^3}{\ell}
\sum_{m \geq 1} m e^{2\pi i m \mathfrak{z}}\sum_{c\geq 1}\frac{K_{\alpha,\gamma}(m,0;c)}{c^2}.
\end{split}
\end{equation}
In particular,
$$
\lim_{z\to \frac{\alpha}{\gamma}} y\Psi_{2,N}(\mathfrak{z},z)=\frac{c_N}{\mathfrak{z}_2} + 2\pi\delta_{\frac{\alpha}{\gamma},\infty}
-\frac{8 \pi^3}{\ell}\sum_{m \geq 1} m e^{2\pi i m \mathfrak{z}}\sum_{c\geq 1}\frac{K_{\alpha,\gamma}(m,0;c)}{c^2}.
$$
\end{lemma}
\begin{proof}
One determines \eqref{eqn:alpha/gammalim} by computing
the Fourier expansion of the analytic continuation to $s=0$ of each of the sums $\sum_1(\alpha,\gamma)$, $\sum_2(\alpha,\gamma)$, and $\sum_3(\alpha,\gamma)$ in the splitting analogous to \eqref{eqn:analcont} for the coset $L^{-1}\Gamma_0(N)$.
The behavior towards the cusp $\alpha/\gamma$ then follows by taking the limit $z\to i\infty$ termwise.  The sum $\sum_1(\alpha,\gamma)$ vanishes unless $L\in \Gamma_0(N)$.  For $\sum_3(\alpha,\gamma)$ we plug $s=0$ into the $(m,n)\neq (0,0)$ terms of \eqref{eqn:sum3cuspexp} and note that by Lemma \ref{lem:53} the contribution from $m=n=0$ is $c_N(\gamma)/\mathfrak{z}_2$, where $c_N(\gamma)$ is some constant.  However, observing that, by Propositions \ref{prop:sum3s}, \ref{prop:sum2s}, and \ref{prop:sum1s}, $\mathfrak{z} \mapsto y\Psi_{2,N}(\mathfrak{z},z)-c_N/\mathfrak{z}_2$ is meromorphic for every $z\in\H$, we conclude that $c_N(\gamma)=c_N$.

Since the proofs of the Fourier expansions for $\sum_2(\alpha,\gamma)$ and for the limit of $\sum_3(\alpha,\gamma)$ at different cusps are similar, we only consider the cusp $i\infty$.
In order to compute the expansions for $\sum_1$ and $\sum_2$, we note that the assumption on $y$ and $\mathfrak{z}_2$ implies that $y>\im(M\mathfrak{z})$ for all $M\in\Gamma_0(N)$.  This follows because, for $M\in\Gamma_{\infty}$, we have $\im(M\mathfrak{z})=\mathfrak{z}_2<y$ and for $M=\left(\begin{smallmatrix}a&b\\c&d\end{smallmatrix}\right)$ with $c\neq 0$, we have
$$
\im(M\mathfrak{z})=\frac{\mathfrak{z}_2}{|c\mathfrak{z}+d|^2}\leq \frac{1}{c^2\mathfrak{z}_2}\leq  \frac{1}{\mathfrak{z}_2}<\frac{1}{v_0}<y-\mathfrak{z}_2<y.
$$
 Hence, for $\sum_1$, we apply Poisson summation to obtain, using \eqref{eqn:gval},
\begin{align}
\label{eqn:sum1exp}{\sum}_1&=2y\sum_{n\in\Z} \frac{1}{(\mathfrak{z}-z+n)(\mathfrak{z}-\overline{z}+n)} = 2y\sum_{n\in\Z} g_{n}\left(\mathfrak{z}-z,\mathfrak{z}-\overline{z}\right)\\
\nonumber &= 2\pi \sum_{n\geq 0} e^{-2\pi i n\mathfrak{z}} e^{2\pi i nz}
+2\pi\sum_{n\geq 1} e^{2\pi in\mathfrak{z}} e^{-2\pi i n\overline{z}}.
\end{align}

We finally consider $\sum_2$, using the representation directly from the definition \eqref{eqn:analcont} with $s=0$ plugged in, namely
$$
{\sum}_2 =2 y\sum_{\substack{M=\left(\begin{smallmatrix}a&b\\c&d\end{smallmatrix}\right)\in \Gamma_0(N)\\  c\geq 1}} \frac{\frac{1}{(M\mathfrak{z}-z)(M\mathfrak{z}-\overline{z})}-\frac{1}{\left(\frac{a}{c}-z\right)\left(\frac{a}{c}-\overline{z}\right)}}{j(M,\mathfrak{z})^2}.
$$
To determine the Fourier expansion of the right-hand side we obtain, using Poisson summation,
\begin{multline*}
2y\sum_{\substack{c\geq 1\\N\mid c}}c^{-2}\sum_{\substack{a,d\pmod{c}\\ ad\equiv 1\pmod{c}}}\sum_{n\in\Z}\sum_{m\in\Z}\\
\times \Bigg(\int_{\R}\frac{e^{-2\pi i mw}}{\left(\mathfrak{z}+\frac{d}{c}+w\right)^2} \int_{\R} \frac{e^{-2\pi i nt}}{\left(-\frac{1}{c^2\left(\mathfrak{z}+\frac{d}{c}+w\right)}+\frac{a}{c}-z+t\right)\left(-\frac{1}{c^2\left(\mathfrak{z}+\frac{d}{c}+w\right)}+\frac{a}{c}-\overline{z}+t\right)}dt dw\\
\left.-g_{n}\left(\frac{a}{c}-z,\frac{a}{c}-\overline{z}\right)g_m\left(\mathfrak{z}+\frac{d}{c},\mathfrak{z}+\frac{d}{c}\right)\right).
\end{multline*}
Using \eqref{eqn:gval}, we evaluate
\begin{align*}
g_n\left(\frac{a}{c}-z,\frac{a}{c}-\overline{z}\right)&=\begin{cases}
\frac{\pi}{y}e^{2\pi in\left(\frac{a}{c}-z\right)} & \text{if }n\leq 0,\\
\frac{\pi}{y}
e^{2\pi in\left(\frac{a}{c}-\overline{z}\right)} & \text{if }
n> 0,
\end{cases}\qquad\qquad\qquad\text{and}\\
g_m\left(\mathfrak{z}+\frac{d}{c},\mathfrak{z}+\frac{d}{c}\right)&=\begin{cases}
0&\text{if }m\leq 0,\\
-4\pi^2me^{2\pi i m\left(\mathfrak{z}+\frac{d}{c}\right)}&\text{if }m>0.
\end{cases}
\end{align*}
\indent
To compute the remaining double integral, we shift $w\mapsto w-\mathfrak{z}-d/c$ and $t\mapsto t-a/c$, so that the double integral becomes
\begin{equation}\label{eqn:doubleint}
e^{\frac{2\pi i \left(na+m d\right)}{c}} e^{2\pi i m\mathfrak{z}} \int_{\R+i\mathfrak{z}_2}\frac{e^{-2\pi i mw}}{w^2} \int_{\R} \frac{e^{-2\pi i nt }}{\left(-\frac{1}{c^2w}-z+t\right)\left(-\frac{1}{c^2w}-\overline{z}+t\right)}dt dw.
\end{equation}
From the restrictions on $y$ and $\mathfrak{z}_2$, one concludes that $-\frac{1}{c^2w}-z\in -\H$  and hence \eqref{eqn:gval} implies that
\rm
 the integral over $t$ equals
$$
g_n\left(-\frac{1}{c^2w}-z,-\frac{1}{c^2w}-\overline{z}\right)=\begin{cases}
\frac{\pi }{y} e^{-2\pi i n\left(\frac{1}{c^2w}+z\right)}&\text{if }n\leq 0,\\
\frac{\pi }{y}
 e^{-2\pi i n\left(\frac{1}{c^2w}+\overline{z}\right)}&\text{if }n> 0.
\end{cases}
$$
Therefore \eqref{eqn:doubleint} equals
\begin{align*}
\frac{\pi}{y}e&^{\frac{2\pi i \left(na+m d\right)}{c}} e^{2\pi i m\mathfrak{z}} e^{-2\pi i n(x-i\sgn(n)y)}\mathcal{I}_{m,n}(\mathfrak{z}_2),\quad\text{where} \quad \mathcal{I}_{m,n}\left(\mathfrak{z}_2\right):=\int_{\R+i\mathfrak{z}_2}\frac{e^{-2\pi i \left(mw+\frac{n}{c^2w}\right)}}{w^2}dw.
\end{align*}

Since the integrand in $\mathcal{I}_{m,n}(\mathfrak{z}_2)$ is meromorphic, the path of integration may be shifted to $\im(w)=\alpha$ for any $\alpha>0$, implying that $\mathcal{I}_{m,n}(\mathfrak{z}_2)=\mathcal{I}_{m,n}(\alpha)$ is independent of $\mathfrak{z}_2$.  Taking the limit $\im(w)\to \infty$ 
yields that $\mathcal{I}_{m,n}(\mathfrak{z}_2)$ vanishes for $m\in-\N_0$. Moreover,
for $n=0$ and $m\in\N$, \eqref{eqn:gval} implies that
$$
\mathcal{I}_{m,n}(\mathfrak{z}_2)=e^{2\pi m\mathfrak{z}_2} g_m\left(i\mathfrak{z}_2,i\mathfrak{z}_2\right)=-4\pi^2m.
$$
\indent
We conclude that the $n=0$ term precisely cancels the product of $g_n$ and $g_m$ computed above.  Finally, for $n\neq 0$ and $m\in\N$, we make the change of variables $w\mapsto i|n|^{1/2} m^{-1/2} c^{-1}w$ and then shift to $\re(w)=\alpha>0$, to obtain
\begin{equation*}
\mathcal{I}_{m,n}\left(\mathfrak{z}_2\right)=-ic\sqrt{\frac{m}{|n|}}\int_{\alpha+i\R}\frac{e^{\frac{2\pi}{c}\sqrt{m|n|}\left(w-\sgn(n)\frac{1}{w}\right)}}{w^2} dw.
\end{equation*}

Using the fact that, for fixed $\mu, \kappa>0$, the functions $t\mapsto(t/\kappa)^{(\mu-1)/2} J_{\mu-1}(2\sqrt{\kappa t})$ and $s\mapsto s^{-\mu} e^{-\kappa/s}$ (resp. $t\mapsto (t/\kappa)^{(\mu-1)/2} I_{\mu-1}(2\sqrt{\kappa t})$ and $s\mapsto s^{-\mu} e^{\kappa/s}$)
 are inverse to each other with respect to the Laplace transform, by (29.3.80) and (29.3.81) of \cite{AS}, then yields
\[
\mathcal{I}_{m,n}\left(\mathfrak{z}_2\right)=2\pi c\sqrt{\frac{m}{|n|}} \times
\begin{cases}
\vspace{1mm} J_1\left(\frac{4\pi\sqrt{mn}}{c}\right)&\quad\text{ if } n>0,\\
I_1\left(\frac{4\pi\sqrt{m|n|}}{c}\right)&\quad\text{ if } n<0.
\end{cases}
\]
\indent
Hence we obtain for $\sum_2$
\begin{multline}\label{eqn:sum2exp}
\sum_2=
4\pi^2 \sum_{\substack{c\geq 1\\ N\mid c}}c^{-1} \sum_{n\geq 1} n^{-\frac{1}{2}} e^{2\pi i n z} \sum_{m\geq 1} m^{\frac{1}{2}}e^{2 \pi i m \mathfrak{z}} K(m,-n;c)\left(I_1\left(\frac{4\pi\sqrt{mn}}{c}\right)
+\frac{2\pi \sqrt{mn}}{c}\right)\\
+ 4\pi^2 \sum_{\substack{c\geq 1\\ N\mid c}}c^{-1} \sum_{n\geq 1 } n^{-\frac{1}{2}} e^{-2\pi i n \overline{z}} \sum_{m\geq 1} m^{\frac{1}{2}}e^{2 \pi i m \mathfrak{z}} K(m,n;c)\left(J_1\left(\frac{4\pi\sqrt{mn}}{c}\right)+\frac{2\pi \sqrt{mn}}{c}
\right).
\end{multline}
\indent
To finish the computation of the Fourier expansion of $y\psi_{2, N}(\mathfrak{z}, z)$, we note that
the expansion of $\sum_3$ is given in Proposition \ref{prop:sum3s}
and the terms $n\neq 0$ precisely cancel the terms appearing after the $I$-Bessel and $J$-Bessel functions.
Combining \eqref{eqn:sum2exp} with \eqref{eqn:analcontsum3} and \eqref{eqn:sum1exp} yields the claimed expansion.

We finally compute the limit of $y\Psi_{2,N}(\mathfrak{z},z)$ as $z\to\alpha/\gamma$, or equivalently, the behavior of $y_L\Psi_{2,N}(\mathfrak{z},z_L)$ as $z\to i\infty$.
For this we take the limit termwise, but to do so we first verify absolute uniform convergence for $y$ sufficiently large.  The Fourier expansion in Proposition \ref{prop:sum3s} converges uniformly in $y>y_0$ for any fixed $y_0>0$, so the contribution to the limit from $\sum_3$ equals
$$
\frac{c_N}{\mathfrak{z}_2} - 8\pi^3\sum_{m\geq 1} m\sum_{\substack{c\geq 1\\ N\mid c}\rm} \frac{K(m,0;c)}{c^2}e^{2\pi i m\mathfrak{z}},
$$
which is a constant with respect to $z$.  The sums $\sum_1$ and $\sum_2$ converge absolutely uniformly under the assumptions given in the lemma, so we may also take the limits $
z\to i\infty
$ termwise; the contribution coming from $\sum_1$ is $2\pi$ and the limit of \eqref{eqn:sum2exp} vanishes,
completing
the proof.
\end{proof}

\subsection{A basis of polar Maass forms}\label{sec:basis}

In this section, we use Proposition \ref{prop:Psigrowth} to construct a family of polar Maass forms with arbitrary principal parts.  For $m\in-\N$ and $\tau_0\in\H$, we let 
\begin{multline}\label{eqn:Ydef}
\mathcal{Y}_{2-2k,m,N}(\tau_0,z):=\frac{i(2i)^{2k-1}}{4v_0\omega_{\tau_0}(-m-1)!}\\
\times
\frac{\partial^{-m-1}}{\partial X_{\tau_0}^{-m-1}(\mathfrak{z})}\!\left[\left(\mathfrak{z}-\overline{\tau_0}\right)^{2k}\!\left(y^{2k-1}\Psi_{2k,N}(\mathfrak{z},z)+\frac{\pi}{3}c_N\delta_{k=1}\widehat{E}_2(\mathfrak{z})\right)\right]_{\mathfrak{z}=\tau_0}.
\end{multline}
\begin{remark}
For $k>1$, Petersson applied his differential operator $\frac{\partial^{-m-1}}{\partial X_{\tau_0}^{-m-1}(\mathfrak{z})}$ to the meromorphic part of $y^{2k-1}\Psi_{2k,N}(\mathfrak{z},z)$ (see (49a) of \cite{Pe2}).  He investigated this function and used the Residue Theorem to compute its principal part in (50) of \cite{Pe2}.
\end{remark}

The following lemma extends Lemma 4.4 of \cite{BKFourier} to include $k=1$ and level $N$.
\begin{lemma}\label{lem:basis}
For $\tau_0\in\H$, $n\in-\N$, and $
k\in\N
$, there exists $F\in \H_{2-2k}^{\operatorname{cusp}}(N)$ with principal part
$$
\left(z-\overline{\tau_0}\right)^{2k-2}X_{\tau_0}^{n}(z) + O\left(\left(z-\tau_0\right)^{n+1}\right)
$$
around $z=\tau_0$ and no other singularities modulo $\Gamma_0(N)$ if and only if $n\equiv k-1\pmod{\omega_{\tau_0}}$.
In particular, if $n\equiv k-1\pmod{\omega_{\tau_0}}$, then the principal part of $z\mapsto \mathcal{Y}_{2-2k,n,N}(\tau_0,z)\in \H_{2-2k}^{\operatorname{cusp}}(N)$ equals $(z-\overline{\tau_0})^{2k-2}X_{\tau_0}^{n}(z)$.
\end{lemma}
\begin{proof}
The necessary condition follows as in the proof of Lemma 4.4 of \cite{BKFourier}.  It hence suffices to show that $z\mapsto\mathcal{Y}_{2-2k,n,N}(\mathfrak{z}, z)$ have prescribed principal parts and are indeed elements of $\H_{2-2k}^{\operatorname{cusp}}(N)$.

The principal parts for $k>1$ are known by (50) of \cite{Pe2}, but a little work is needed to translate the calculation into a statement useful for our purposes.  Petersson technically only computed the principal part coming from acting with his differential operator on a function $H_{2k,N}(\mathfrak{z},z)$.    However, in Proposition 3.1 of \cite{BKFourier} it was shown that $H_{2k,N}$ is the meromorphic part of $(2iy)^{2k-1}\Psi_{2k,N}(\mathfrak{z},z)$.  Furthermore, for $k>1$, we also have $y^{2k-1}\Psi_{2k,N}(\mathfrak{z},z)\in \H_{2-2k}^{\operatorname{cusp}}(N)$ by Proposition 3.1 
of \cite{BKFourier}.  Since
\rm
 acting by a differential operator in the independent variable $\mathfrak{z}$ preserves both modularity and harmonicity in $z$, one easily concludes that $\mathcal{Y}_{2-2k,n,N}\in \H_{2-2k}^{\operatorname{cusp}}(N)$.  Since the non-meromorphic part of $z\mapsto(2iy)^{2k-1}\Psi_{2k,N}(\mathfrak{z},z)$ is real analytic, its image under Petersson's differential operator is also real analytic, and hence does not contribute to the principal part.  To conclude the claim for $k>1$ we thus only need to plug in (50) of \cite{Pe2}, which we next rewrite in our notation for the reader's convenience.\\
\indent
Since Petersson's function $z\mapsto H_{2k,N}(\mathfrak{z}, z)$ is not modular, it was necessary for him to compute the principal part around $z=L\tau_0$ separately for each $L\in\Gamma_0(N)$.
Due to the modularity of $y\Psi_{2,N}(\mathfrak{z},z)$, in order to determine the principal parts at all such points, it suffices to compute the principal part at  $z=\tau_0$.  Translating 
(50) of \cite{Pe2}
 into the notation in this paper and specializing to the case we are looking at, we have $l=\omega_{\tau_0}$, $\omega=\omega_1=\tau_0$, $t=X_{\tau_0}(z)$, and $\eta_1=q_1=1$.  If the appropriate congruence condition is satisfied, then by (50) of \cite{Pe2} the principal part of $\mathcal{Y}_{2-2k,n,N}(\tau_0,z)$ around $z=\tau_0$ equals
$$
\frac{i}{4v_0\omega_{\tau_0}}\left(-2\omega_{\tau_0}(2iv_0)(z-\overline{\tau_0})^{2k-2}X_{\tau_0}^{n}(z)\right)= (z-\overline{\tau_0})^{2k-2}X_{\tau_0}^{n}(z).
$$

We next turn to the case $k=1$.  By Lemma \ref{lem:xiPsi}, we have $y\Psi_{2,N}(\mathfrak{z},z)\in \H_0^{\operatorname{cusp}}$. Furthermore, the principal parts of $y\Psi_{2,N}(\mathfrak{z},z)$ are the same as for $k>1$ by Proposition \ref{prop:Psigrowth}.  Now note that
\begin{equation}\label{eqn:YnY1}
\mathcal{Y}_{0,n,N}\left(\tau_0,z\right)=
\frac{1}{v_0\omega_{\tau_0}}
 \frac{1}{(-n-1)!}\frac{\partial^{-n-1}}{\partial X_{\tau_0}^{-n-1}(\mathfrak{z})}\left[ -\frac{\omega_{\mathfrak{z}}}{4\mathfrak{z}_2}\left(\mathfrak{z}-\overline{\tau_0}\right)^{2} \mathcal{Y}_{0,-1,N}(\mathfrak{z},z)\right]_{\mathfrak{z}=\tau_0}.
\end{equation}

Since $\mathcal{Y}_{0,-1,N}$ is meromorphic as a function of $\mathfrak{z}$, the computation of the principal parts of $\mathcal{Y}_{0,n,N}$ follows the proof of (50) in \cite{Pe2}, except that one must be careful to verify that, for each $n\in-\N$, $\frac{\partial^{-n-1}}{\partial X_{\tau_0}^{-n-1}(\mathfrak{z})}[[\sum_{3}]_{s=0}]_{\mathfrak{z}=\tau_0}$ does not contribute to the principal part, where $[\sum_{3}]_{s=0}$ denotes the analytic continuation to $s=0$ of ${\sum}_3$ from \eqref{eqn:analcont}.  Acting by Petersson's differential operator termwise and noting absolute locally uniform convergence due to exponential decay in $y$ and $\mathfrak{z}_2$, one easily determines that the resulting Fourier expansion converge for every $\mathfrak{z}, z\in\H$ and hence these terms do not contribute to the principal parts.  This completes the proof.

\end{proof}

Lemma \ref{lem:basis} then yields the following proposition.
\begin{proposition}\label{prop:basis}
For each choice of $\tau_1,\dots,\tau_r\in\H$ and $k\geq 1$, there exists $F\in \H_{2-2k}^{\operatorname{cusp}}(N)$ with principal parts in $\H$ given by
$$
\sum_{d=1}^r\left(z-\overline{\tau_{d}}\right)^{2k-2}\sum_{\substack{ n<0\\
n\equiv k-1\pmod{\omega_{\tau_{d}}}}} b_{\tau_{d}}(n) X_{\tau_{d}}^{n}(z)
$$
and principal part at each cusp $\varrho$ given by  $\sum_{n< 0}a_{\varrho}(n) e^{\frac{2\pi i nz}{\ell_{\varrho}}}.$  Furthermore, $F$ is unique up to addition by a constant (for $k=1$) and is explicitly given by
$$
\sum_{\varrho\in\mathcal{S}_N} \sum_{n<0} a_{\varrho}(n) \mathcal{P}_{2-2k,n,N}^{\varrho}(z)+ \sum_{d=1}^{r} \sum_{\substack{n<0\\ n\equiv k-1\pmod{\omega_{\tau_{d}}}}} b_{\tau_{d}}(n) \mathcal{Y}_{2-2k,n,N}(\tau_{d},z).
$$

\end{proposition}
\begin{proof}
The existence of $F$ follows directly by Lemma \ref{lem:basis}, and the uniqueness comes from the fact that the only harmonic Maass forms with trivial principal parts are holomorphic modular forms,
which follows by Proposition 3.5 of \cite{BF}.
\end{proof}

\subsection{Differentials}

In this section, we consider the properties of $\mathcal{Y}_{0,-1,N}$.  Using the connection between weight $2$ forms and differentials, we prove Theorem \ref{thm:differentials}.

\begin{proof}[Proof of Theorem \ref{thm:differentials}] (1) We have to show that $y\Psi_{2,N}(\mathfrak{z},z)$ has precisely a simple pole at $\mathfrak{z}=z$.  As in the proof of Proposition \ref{prop:Psigrowth} for non-elliptic fixed points, these correspond to the terms from $\sum_1$ and $\sum_2$ for which $M\mathfrak{z}=z$, since $\sum_3$ converges absolutely in $\H$.  One sees directly that the poles are at most simple and the only matrices contributing to the pole at $\mathfrak{z}=z$ are $M=\pm I$, yielding the principal part $-i/(\mathfrak{z}-z)$.  For elliptic fixed points, the argument is similar, except that we must make sure that the congruence conditions for the coefficients are satisfied.  The congruence condition is $-1\equiv -k\pmod{\omega_{z}}$, which in this case ($k=1$) is 
always satisfied. 

\noindent
(2)  By Lemma \ref{lem:xiPsi} and the fact that $\widehat{E}_{2}(\mathfrak{z})$ in annihilated by $\xi_{0,z}$, we
obtain the cusp form
$$
\xi_{0,z}\left(y\Psi_{2,N}(\mathfrak{z},z)\right)=\Phi_N(z,\mathfrak{z}).
$$
\noindent
(3)  We may differentiate \eqref{eqn:analcontsum3}, \eqref{eqn:sum20},
and \eqref{sum10} directly.   There is no contribution to the residue at $\mathfrak{z}=z$ from \eqref{eqn:analcontsum3} because it is real analytic.  Since
$$
z\mapsto \frac{1}{\left(M\mathfrak{z}-z\right)\left(\frac{a}{c}-z\right)}\qquad \text{ and }\qquad z\mapsto \frac{1}{\mathfrak{z}-n-z}
$$
are both meromorphic, they have Laurent expansions around $z=\mathfrak{z}$, and hence their derivatives have trivial residue, as asserted here.
\end{proof}

\section{Bruinier--Funke pairing and the proofs of Theorem \ref{Peterssonch} and Corollary \ref{cor:Fourier}}\label{sec:mainproof}

In this section, we use a slight variant of the Bruinier--Funke pairing \cite{BF} to finish the proof of Theorem \ref{Peterssonch}.  For $g\in S_{2k}(N)$ and  $F\in \mathcal{H}_{2-2k}^{\operatorname{cusp}}(N),
$
 define the \textit{pairing}
\[
\{g, F\}:=\left(g, \xi_{2-2k}(F)\right),
\]
where, for $g,h\in S_{2k}(N)$ and $d\mu:=\frac{dx dy}{y^2}$,
$$
(g,h):=\frac{1}{\mu_N}\int_{\Gamma_0(N)\backslash \H} g(z)\overline{h(z)} y^{2k} d\mu
$$
is the standard \begin{it}Petersson inner-product\end{it} with $\mu_N:=[\SL_2(\Z):\Gamma_0(N)]$.  The pairing $\{g,F\}$ was computed for $F\in H_{2-2k}^{\operatorname{cusp}}(N)$ by Bruinier and Funke in Proposition 3.5 of \cite{BF}.  The following proposition extends their evaluation of $\{g,F\}$ to the entire space $\mathcal{H}_{2-2k}^{\operatorname{cusp}}(N)$.
\begin{proposition}\label{expansionprod}
Assume that $g\in S_{2k}(N)$ and write its elliptic expansion around each $\mathfrak{z}\in\H$ as $g(z)=(z-\overline{\mathfrak{z}})^{-2k}\sum_{n\geq 0}a_{g,\mathfrak{z}}(n)X_{\mathfrak{z}}^n(z)$  and its expansion at each cusp $\varrho$ as $g_{\varrho}(z) =\! \sum_{n\geq 1}a_{g,\varrho}(n)e^{\frac{2\pi i nz}{\ell_{\varrho}}}$.  Suppose that $F\in\mathcal{H}_{2-2k}^{\operatorname{cusp}}(N)$ and
write its expansion around $\mathfrak{z}\in\H$ as
\begin{equation*}
F(z)=\left(z-\overline{\mathfrak{z}}\right)^{2k-2}\sum_{n\gg -\infty} b_{F,\mathfrak{z}} (n)X_{\mathfrak{z}}^n(z)+\left(z-\overline{\mathfrak{z}}\right)^{2k-2}\sum_{n\leq -1} c_{F,\mathfrak{z}}(n)\beta \left(r_{\mathfrak{z}}^2(z); -n,2k-1\right)X_{\mathfrak{z}}^n(z).
\end{equation*}
Moreover, write the Fourier expansion of $F$ at the cusp $\varrho$ as
$$
F_\varrho(z)= \sum_{n \gg -\infty}c_{F,\varrho}^+(n)e^{\frac{2\pi i nz}{\ell_{\varrho}}} + \sum_{n < 0}c_{F,\varrho}^-(n) \Gamma\left(2k-1,\frac{4\pi|n|y}{\ell_{\varrho}}\right)e^{\frac{2\pi i nz}{\ell_{\varrho}}}.
$$
Then we have
\begin{equation}\label{eqn:gFpair}
\{g, F\}=
\frac{\pi}{\mu_N}
\sum_{\mathfrak{z}\in\Gamma_0(N)\backslash\H} \frac{1}{\mathfrak{z}_2\omega_\mathfrak{z}}\sum_{n\geq 1
}b_{F,\mathfrak{z}} \left(-n\right) a_{g,\mathfrak{z}}\left(n-1\right) +
\frac{1}{\mu_N}
\sum_{\varrho\in\mathcal{S}_N} \sum_{
n\geq 1} c_{F,\varrho}^{+}(-n)a_{g,\varrho}(n).
\end{equation}
\end{proposition}
\begin{proof}

The proof closely follows the proof of Proposition 3.5 in \cite{BF}.   By Proposition \ref{prop:basis} and linearity, $\mathcal{H}_{2-2k}^{\operatorname{cusp}}(N)$ decomposes into the direct sum of $H_{2-2k}^{\operatorname{cusp}}(N)$ and $\H_{2-2k}^{\operatorname{cusp}}(N)$.
Each of these subspaces then further
splits into direct sums of subspaces consisting of forms
with poles at precisely one cusp $\varrho$ or precisely one point $\mathfrak{z}\in \Gamma_0(N)\backslash\H$.  Hence it suffices to assume that $F$ has a pole at exactly one $\mathfrak{z}\in\Gamma_0(N)\backslash\H$.

We may choose a fundamental domain $\mathcal{F}(N)$ of $\Gamma_0(N)\backslash\H$ for which $
\mathfrak{z}$ is in the interior of $\Gamma_{\mathfrak{z}}\mathcal{F}(N) = \left\{ Mz\middle| M\in \Gamma_{\mathfrak{z}}, z\in\mathcal{F}(N)\right\}$.   Since $\Gamma_{\mathfrak{z}}$ fixes $\mathfrak{z}$, there is precisely one copy of $\mathfrak{z}$ modulo $\Gamma_0(N)$ inside $\Gamma_{\mathfrak{z}}\mathcal{F}(N)$.  Denoting the ball around $\mathfrak{z}$ with radius $\varepsilon>0$ by $B_\varepsilon (\mathfrak{z})$ and noting that $\overline{\xi_{2-2k}(F)}=y^{-2k}L_{2-2k}(F)$, we may rewrite
\[
\{g, F\}=\lim_{\varepsilon\to 0}\frac{1}{\mu_N\omega_{\mathfrak{z}}}\int_{\Gamma_{\mathfrak{z}}\mathcal{F}(N)\setminus B_{\varepsilon}(\mathfrak{z})}g(z) L_{2-2k}\left(F(z)\right)d\mu.
\]
Using $d(g(z)\overline{F(z)}dz)=-g(z)L_{2-2k}(F(z))d\mu,$ and applying Stokes' Theorem, the integral becomes
\begin{equation}\label{eqn:Stokes}
\frac{1}{\mu_N\omega_{\mathfrak{z}}}\int_{\partial B_{\varepsilon}(\mathfrak{z})}g(z)F(z)dz,
\end{equation}
where the integral is taken counter-clockwise.  We next insert the expansions of $F$ and $g$  around $z=\mathfrak{z}$.  Since the non-meromorphic part of $F$ is real analytic (and hence in particular has no poles) in $\H$, we may directly plug in $\varepsilon=0$ to see that the contribution from non-meromorphic part to \eqref{eqn:Stokes} vanishes.  Therefore, the limit $\varepsilon\to 0$ of \eqref{eqn:Stokes} equals the limit $\varepsilon\to 0$ of
\[
\frac{1}{\mu_N\omega_{\mathfrak{z}}}\sum\limits_{\substack{n\gg-\infty\\ m\geq 0}}b_{F,\mathfrak{z}}(n)a_{g,\mathfrak{z}}\left(m\right)
\int_{\partial B_{\varepsilon}\left(\mathfrak{z}\right)}
\left(z-\overline{\mathfrak{z}}\right)^{-2-\left(n+m\right)}(z-\mathfrak{z})^{n+m}dz.
\]
Setting  $n+m=-\ell$, we obtain, by the Residue Theorem,
\begin{equation}\label{eqn:ResBF}
\int_{\partial B_\varepsilon(\mathfrak{z})}
 \left(z-\overline{\mathfrak{z}}\right)^{\ell-2} (z-\mathfrak{z})^{-\ell}dz=2\pi i\mathop{\mathrm{Res}}\limits_{z=\mathfrak{z}}\left(\left(z-\overline{\mathfrak{z}}\right)^{\ell-2}(z-\mathfrak{z})^{-\ell}\right).
\end{equation}
From this, one sees immediately that
there is no contribution from $\ell\leq 0
$.
If $\ell\geq 2$ we get, writing $z-\overline{\mathfrak{z}}=z-\mathfrak{z}+2i\mathfrak{z}_2$,
\[
\left(z-\overline{\mathfrak{z}}\right)^{\ell-2}=\sum_{j=0}^{\ell-2}\binom{\ell-2}{j}\left(2i\mathfrak{z}_2\right)^j (z-\mathfrak{z})^{\ell-2-j},
\]
thus the residue in \eqref{eqn:ResBF} vanishes.
 If $\ell=1$, then we obtain
\[
(z-\mathfrak{z})^{-1}\left(z-\overline{\mathfrak{z}}\right)^{-1}=\frac{1}{2i\mathfrak{z}_2
(z-\mathfrak{z})\left(1+\frac{z-\mathfrak{z}}{2i\mathfrak{z}_2}\right)}=\frac1{2i\mathfrak{z}_2}\sum_{j=0}^\infty \left(-2i\mathfrak{z}_2\right)^{-j}(z-\mathfrak{z})^{-1+j}.
\]
So only the term with $j=0$ contributes to the residue, giving the claim.
\end{proof}

We are finally ready to extend Satz 3 of \cite{Pe1} to include the case $k=1$.
\begin{proof}[Proof of Theorem \ref{Peterssonch}]
By Proposition \ref{prop:basis}, the linear combination \eqref{eqn:genpolar} is a  (unique, up to addition by a constant if $k=1$) weight $2-2k$ polar harmonic Maass form $F\in \H_{2-2k}^{\operatorname{cusp}}(N)$ with the principal parts as in the theorem.  Since $F\in\mathcal{M}_{2-2k}(N)$ if and only if $\xi_{2-2k}(F)=0$ and $\xi_{2-2k}(F)$ is a cusp form, we conclude that $F\in\mathcal{M}_{2-2k}(N)$ if and only if $\xi_{2-2k}(F)$ is orthogonal to every cusp form $g\in S_{2k}(N)$.  However, by Proposition \ref{expansionprod}, we see that this occurs if and only if \eqref{eqn:gFpair} holds for every $g\in S_{2k}(N)$.  This is the statement of the theorem.
\end{proof}

We are now ready to give an explicit version of Corollary \ref{cor:Fourier}, which is an easy consequence of Theorem \ref{Peterssonch} together with Lemma \ref{lem:alpha/gammalim}.
To state the theorem, we require the sum-of-divisors function $\sigma(m):=\sum_{d\mid m} d$.
\begin{theorem}\label{thm:Fourier}
Suppose that $\tau_1,\dots, \tau_r\in\H$ are given such that
$$
F(z):=\sum_{\varrho\in\mathcal{S}_N} \sum_{n<0} a_{\varrho}(n) \mathcal{P}_{0,n,N}^{\varrho}(z) + \sum_{d=1}^{r} \sum_{\substack{n\equiv 0\pmod{\omega_{\tau_{d}}}\\ n<0}} b_{\tau_{d}}(n) \mathcal{Y}_{0,n,N}\left(\tau_{d},z\right)
$$
satisfies \eqref{eqn:Peterssonch} (with $k=1$).  For $y$ sufficiently large (depending on $v_1,\dots,v_d$), we have the expansion
\begin{multline*}
F(z)=\sum_{\varrho=\frac{\alpha}{\gamma}\in\mathcal{S}_N} \sum_{n<0} a_{\varrho}(n)\left(\vphantom{\substack{b\vspace{0.5in}}}\right. \delta_{\varrho,\infty} e^{2\pi i n z}+ \frac{4\pi^2}{\ell_{\varrho}} |n|\sum_{c\geq 1} \frac{K_{\alpha,\gamma}(n,0;c)}{c^3}\\
\hfill + 2\pi \left(\frac{|n|}{\ell_{\varrho}}\right)^{\frac{1}{2}} \sum_{j\geq 1}j^{-\frac{1}{2}}  \sum_{c\geq 1} \frac{K_{\alpha,\gamma}(n,j;c)}{c}I_1\left(\frac{4\pi}{c} \sqrt{\frac{|n|j}{\ell_{\varrho}}}\right) e^{2\pi i jz}\left.\vphantom{\substack{b\vspace{0.5in}}}\right)\\
-  \sum_{d=1}^{r} \sum_{\substack{n\equiv 0\pmod{\omega_{\tau_{d}}}\\ n<0}} b_{\tau_{d}}(n)
\frac{\pi}{v_d\omega_{\tau_d}}\frac{1}{(-n-1)!}\frac{\partial^{-n-1}}{\partial X_{z}^{-n-1}(\mathfrak{z})} \left[\vphantom{\substack{b\vspace{0.6in}}}\right.
\left(\mathfrak{z}-\overline{\tau_d}\right)^{2} \sum_{j\geq 0} e^{2\pi i jz}e^{-2\pi i j\mathfrak{z}} \hfill \\
+ 2\pi\left(\mathfrak{z}-\overline{\tau_d}\right)^{2}
  \sum_{j\geq 1} j^{-\frac{1}{2}} e^{2\pi i j z} \sum_{m\geq 1} m^{\frac{1}{2}}e^{2 \pi i m \mathfrak{z}} \sum_{\substack{c\geq 1\\
N\mid c}}c^{-1} K(m,-j;c)I_1\left(\frac{4\pi\sqrt{m j}}{c}\right)\\
- 4 \pi^2\left(\mathfrak{z}-\overline{\tau_d}\right)^{2}\sum_{m\geq 1} m e^{2\pi i m \mathfrak{z}}\sum_{\substack{c\geq 1\\ N\mid c}}\frac{K(m,0;c)}{c^2} -\frac{4\left(\mathfrak{z}-\overline{\tau_d}\right)^{2}
}{N \prod_{p|N} \left(1+p^{-1}\right)
} \left(1-24\sum_{m\geq 1} \sigma(m)e^{2\pi i m\mathfrak{z}}\right)\left.\vphantom{\substack{b\vspace{0.6in}}}\right]_{\mathfrak{z}=\tau_d}.
\end{multline*}
\end{theorem}
\begin{proof}

The claim follows by computing the Fourier expansions of $\mathcal{P}_{0,n,N}^{\varrho}$ and $\mathcal{Y}_{0,n,N}$ for each $n\in -\mathbb{N}$.  However, since $F$ is meromorphic by Theorem \ref{Peterssonch},
we only need to compute the meromorphic parts of each Fourier expansion.  We
begin by plugging in the meromorphic parts of the the Fourier expansions of $\mathcal{P}_{0,n,N}^{\varrho}$ given in  Theorem \ref{thm:weakMaass}.

In order to compute the Fourier expansions of the $\mathcal{Y}_{0,n,N}$, we assume that $y$ is sufficiently large so that in particular
there exists $v_0>0$ satisfying $2v_0<v_{d}<y-1/v_0$ for every $d\in\{1,\dots,r\}$.
By \eqref{eqn:YnY1},  to determine the expansions of $\mathcal{Y}_{0,n,N}^{+}$, we only need to apply the differential operators in the definition \eqref{eqn:Ydef} to the expansion
of $\mathcal{Y}_{0,-1,N}^+$.  Furthermore, the expansion of $\mathcal{Y}_{0,-1,N}^+$ may be directly obtained by taking the meromorphic part of the expansion given in Lemma \ref{lem:alpha/gammalim} plus
$$
\frac{\pi}{3} c_N \widehat{E}_2(\mathfrak{z})=-\frac{c_N}{\mathfrak{z}_2}+ \frac{4\pi}{N}\prod_{p\mid N}\left(1+p^{-1}\right)^{-1}\left(1-24\sum_{
m\geq 1
}\sigma(m)e^{2\pi i m\mathfrak{z}}\right).
$$
This yields the statement of the theorem.
\end{proof}

\end{document}